\documentclass[11pt,oneside,letterpaper]{amsart}
\usepackage{amssymb}

\usepackage{amsfonts,amsmath,amstext,amsbsy,euscript,amsthm}
\usepackage{mathtools}
\usepackage{enumerate}
\usepackage{mathabx}
\usepackage[all]{xy}
\usepackage[utf8]{inputenc} 
\usepackage{color}
\usepackage{tikz-cd}
\usepackage{mathrsfs}
\usepackage[colorlinks]{hyperref}

\DeclarePairedDelimiter\ceil{\lceil}{\rceil}
\DeclarePairedDelimiter\floor{\lfloor}{\rfloor}

\theoremstyle{plain}
\newtheorem{theorem}{Theorem}[section]
\newtheorem{proposition}[theorem]{Proposition}

\newtheorem{lemma}[theorem]{Lemma}

\theoremstyle{definition}
\newtheorem{definition}[theorem]{Definition}

\newtheorem{question}[theorem]{Question}

\theoremstyle{remark}
\newtheorem{remark}[theorem]{Remark}
\newcommand{\note}[2][\null]{%
  \marginpar{\renewcommand{\baselinestretch}{1}\vspace{-1em}\hrule\vspace{3pt}%
  \scriptsize\raggedright\textsf{#2\ifx#1\null\else\\\hfill--- 
  {\em #1}\fi}\vspace{1.5em}}%
}

\numberwithin{equation}{section}
\title{The Snapshot Problem for the Wave Equation}
\author{Fulton Gonzalez}
\address{Department of Mathematics,
Tufts University, 
Medford, MA 02155, USA}
\email{fulton.gonzalez@tufts.edu}
\author{Jue Wang}
\address{College of Science, North China Institute of Science \& Technology, Langfang, Hebei, 065201, China}
\email{wangjue\textunderscore math@126.com}
\author{Tomoyuki Kakehi}
\address{Department of Mathematics,
University of Tsukuba,
Tsukuba, Ibaraki, 305-8571, Japan}
\email{kakehi@math.tsukuba.ac.jp}
\author{Jens Christensen}
\address{Department of Mathematics, 
Colgate University,
Hamilton, NY 13346, USA}
\email{jchristensen@colgate.edu}
\date{\today}

\begin{document}

\maketitle

\begin{abstract}
By definition, a \emph{wave} is a $C^\infty$ solution $u(x,t)$ of the wave equation on $\mathbb R^n$, and a \emph{snapshot} of the wave $u$ at time $t$ is the function $u_t$ on $\mathbb R^n$ given by $u_t(x)=u(x,t)$.  We show that there are infinitely many waves with given $C^\infty$ snapshots $f_0$ and $f_1$ at times $t=0$ and $t=1$ respectively, but that all such waves have the same snapshots at integer times.  We present a necessary condition for the uniqueness, and a compatibility condition for the existence, of a wave $u$ to have three given snapshots at three different times, and we show how this compatibility condition leads to the problem of small denominators and Liouville numbers.  We extend our results to shifted wave equations on noncompact symmetric spaces.  Finally, we consider the two-snapshot problem and corresponding small denominator results for the shifted wave equation on the $n$-sphere. 
\end{abstract}

\section{Introduction}

Consider the wave equation for $u(x,t)\in C^\infty(\mathbb R^n\times\mathbb R)$ given by
\begin{equation}\label{E:wave-eqn1}
\Delta u=\frac{\partial^2 u}{\partial t^2},
\end{equation}
where $\Delta=\sum_{j=1}^n \partial^2/\partial x_j^2$.  Suppose that  $u$ has Cauchy data
\begin{equation}\label{E:cauchydata}
u(x,0)=f(x),\quad \frac{\partial u}{\partial t}(x,0)=g(x)\qquad (x\in\mathbb R^n),
\end{equation}
where $f,\,g\in C^\infty(\mathbb R^n)$.  Then $u$ can be written in the well-known form
\begin{equation}\label{E:conv-soln1}
    u(x,t)=(f*S_t')(x)+(g*S_t)(x),
\end{equation}
where $S_t$ and $S_t'$ are compactly supported distributions on $\mathbb R^n$ whose Fourier-Laplace transforms are given by
\begin{equation}\label{E:wave-fund-soln}
\begin{aligned}
    \frac{\sin(t\,\sqrt{\zeta\cdot\zeta})}{\sqrt{\zeta\cdot\zeta}}&=\int_{\mathbb R^n} e^{-i\,\zeta\cdot x}\,dS_t(x), \\
	\cos(t\, \sqrt{\zeta\cdot\zeta})&=\int_{\mathbb R^n} e^{-i\,\zeta\cdot x}\,dS_t'(x)\qquad (\zeta\in\mathbb C^n).
\end{aligned}
\end{equation}
(See, for example, \cite{HelgasonMultitemporal} for details. We can also write the first function above as $t\,\text{sinc}(t\sqrt{\zeta\cdot\zeta})$.)  The distributions $S_t$ and $S_t'$ are rotation-invariant and have support in the closed ball $\overline B_{|t|}(0)$, since for fixed $t$ the left hand sides of \eqref{E:wave-fund-soln} are entire functions on $\mathbb C^n$ of exponential type $|t|$. (See \cite{Ho2}, Theorem 16.3.10.)  Following \cite{HelgasonMultitemporal}, we call the family of distribution-pairs $(S_t,\,S_t')_{t\in\mathbb R}$ the \emph{fundamental solution} to the wave equation \eqref{E:wave-eqn1}.

By definition, a \emph{wave} in $\mathbb R^n$ is a solution of the wave equation \eqref{E:wave-eqn1}.  Throughout this article, we will assume that our waves are $C^\infty$ functions of $(x,t)\in\mathbb R^n\times \mathbb R$. (This will certainly be the case when the Cauchy data \eqref{E:cauchydata} are $C^\infty$.) For any fixed $t\in\mathbb R$, we define the \emph{snapshot of the wave $u$ at time $t$} to be the function $u_t$ on $\mathbb R^n$ given by
\begin{equation}\label{E:snapshot-def}
u_t(x)=u(x,t)\qquad (x\in\mathbb R^n).
\end{equation}

It is of interest to determine whether the Cauchy data \eqref{E:cauchydata} can be replaced by snapshot data in determining the wave $u$. To be more precise, suppose that we are given $C^\infty$ functions $f_0(x)$ and $f_1(x)$ on $\mathbb R^n$. Does there exist a wave $u(x,t)$ such that $u_0=f_0$ and $u_1=f_1$?  (There is no loss of generality in considering snapshots at times $t=0$ and $t=1$.) And if such a $u$ exists, is it unique?

The above questions were considered for waves in $\mathbb R^3$ by Fritz John in his 1955 book \emph{Plane Waves and Spherical Means Applied to Partial Differential Equations} (\cite{JohnPlaneWaves}). In the book, John used elementary methods involving mean value operators to show that such a wave $u$ exists. John also showed that while the wave $u$ is not unique, the integer time snapshots $u_k$ are uniquely determined by $u_0$ and $u_1$ for any $k\in\mathbb Z$.

One of the aims of this paper is to generalize John's results to wave equations on Euclidean spaces of arbitrary dimension. Another main aim is to formulate necessary and sufficient conditions for a wave $u(x,t)$ to be determined uniquely by three snapshots.  In the latter part of the paper, we will also consider the snapshot problem for shifted wave equations on noncompact symmetric spaces and on spheres.

The paper is organized as follows.  In Section 2, we introduce Ehrenpreis' notions of invertible distributions and slowly decreasing functions to show, in Proposition \ref{T:two-snapshot-existence}, 
that there exist infinitely many waves with two given snapshots at times $t=0$ and $t=1$.  In Theorem \ref{T:integersnapshot}, basic trigonometric identities are then used to obtain a formula for the snapshot at any integer time.

Section 3 is the main part of the paper.  In it we consider the existence and uniqueness question for a wave to have three given snapshots at three different times, which can be taken to be $t=0,\,1$, and $\alpha$. Uniqueness is shown to be a straightforward consequence of the recursion formula in Theorem \ref{T:integersnapshot2}. As for existence, it is shown in \eqref{E:compatibility}--\eqref{E:compatibility-three} that three given $C^\infty$ functions on $\mathbb R^n$ can be snapshots of a wave $u(x,t)$ at three different times only if they satisfy a compatibility condition expressible in terms of convolutions.  The question of whether this compatibility condition is sufficient leads to the problem of ``small denominators.'' Specifically, if $\alpha$ is an irrational number, the sufficiency of this compatibility condition  boils down to whether the sequence $k\mapsto \sin \pi k\alpha$ does not have a rapidly decreasing subsequence, which in turn depends on whether or not $\alpha$ is a \emph{Liouville number}; i.e., a number that in a precise sense can be closely approximated by rational numbers.
(See Definition \ref{D:liouville-number} below.)

The problem of small denominators has a long pedigree, starting more than a century ago in the study of the $n$-body problem, and continuing with important insights by Kolmogorov, Arnold, Moser, and others in the 1960s.  See Yoccoz's paper \cite{Yoccoz1992} for a readable introduction to the subject and a useful list of references.

The main results in Section 3 are as follows.  When $\alpha$ is irrational, Theorems \ref{T:liouville2}  and \ref{T:irrational-not-Liouville} use functional analytic arguments to show that the compatibility condition guarantees the existence of a unique wave with given snapshots at times $t=0,1,\alpha$ if and only if $\alpha$ is not a Liouville number.  The ``if'' part invokes a result of Ehrenpreis on ideals of spaces of holomorphic functions in $\mathbb C^n$ generated by functions which are jointly slowly decreasing.

When $\alpha$ is rational, we will show in Theorem \ref{T:psi-compatibility} that a stronger set of compatibility conditions are both necessary and sufficient for a wave to have snapshots $f_0,\,f_1,\,f_\alpha$ at times $t=0,1,\alpha$.

In Section 4, we extend our results to shifted wave equations on noncompact symmetric spaces. Some care is needed to adapt our Euclidean results since the harmonic analysis on such spaces is no longer commutative.  Such equations were studied in various contexts by Semenov-Tian-Shansky \cite{Semenov-Tian-Shansky}, Helgason \cite{HelgasonHuygens}, Branson-\'Olafsson-Pasquale \cite{Olafsson2005}, Hassani \cite{Hassani2011}, and others.

Finally, in Section 5, we consider the two-snapshot problem for the shifted wave equation on the sphere $S^n$, and its relation to the problem of small denominators. This type of wave equation on the sphere was introduced by Lax and Phillips in 1978 (\cite{LaxPhillips1978}), who used the easily obtained series solution to deduce Huygens' principle.  They then employ a conformal map to deduce Huygens' principle for the standard wave equation in $\mathbb R^n$.

We prove in Theorem \ref{T:snapshot-sphere-uniqueness} and Theorem \ref{T:sphere-snapshot-existence} that there is a unique wave $u(x,t)$ on the sphere $S^n$ with given snapshots at times $t=0$ and $t=\alpha$ if and only if $\alpha/\pi$ is irrational but not a Liouville number in the case when $n$ is odd.  When $n$ is even, existence and uniqueness occur if and only if either $\alpha/2\pi$  is irrational but not a Liouville number of odd type, or $\alpha/\pi$ is rational with odd numerator when written in lowest terms. Here the issue of small denominators occurs again as a consequence of the rate of decay of certain Schur constants.  A similar problem (in the context of Sobolev spaces) with shifted Cauchy data and concomitant small denominator problem was considered by Rubin in 2000 (\cite{RubinSmallDenom}).

\emph{Notation.} We use $\mathscr E(\mathbb R^n)$ and $\mathscr D(\mathbb R^n)$ to denote the vector spaces $C^\infty(\mathbb R^n)$ and $C_c^\infty(\mathbb R^n)$, equipped with their usual locally convex Frech\'et and inductive limit topologies, respectively.  Their dual spaces $\mathscr E'(\mathbb R^n)$ and $\mathscr D'(\mathbb R^n)$ can be endowed either with their strong or weak* topologies, as this will not be important in the mathematical content of this paper.  

If $X$ is any smooth manifold, the locally convex spaces $\mathscr E(X)$ and $\mathscr D(X)$ are defined similarly.  See \cite{GGA}, Ch.~II, \S2 for details about their topology.

\section{Convolution Operators and Waves with Two Given Snapshots}

For a (compactly supported) distribution $T\in\mathscr E'(\mathbb R^n)$, let $C_T\colon \mathscr E(\mathbb R^n)\to\mathscr E(\mathbb R^n),\;f\mapsto f*T$ be the associated convolution operator.  It is a subject of great interest (and the corresponding literature is vast) to determine the elements of the kernel of $C_T$, which are called \emph{mean periodic functions} with respect to $T$.  It is also of interest to determine the range of $C_T$, and in particular to determine whether $C_T$ is surjective.  For the latter question, a necessary and sufficient condition for the surjectivity of $T$ was obtained by Ehrenpreis in 1960 \cite{Ehrenpreis1960}.  To state this condition properly, we will need to define the notion of a \emph{slowly decreasing function}.

\begin{definition}\label{D:slow-decrease}
A function $F$ on $\mathbb C^n$ is said to be \emph{slowly decreasing} provided that there exists a constant $A>0$ such that for any point $\xi\in\mathbb R^n$, there is a point $\eta\in\mathbb C^n$ for which $|\eta-\xi|<A\,\log(2+|\xi|)$ and
\begin{equation}\label{E:slow-decrease1}
|F(\eta)|\geq (A+|\xi|)^{-A}.
\end{equation}
\end{definition}
Suppose that $F$ is entire in $\mathbb C^n$ and is the Fourier-Laplace transform of a compactly supported distribution on $\mathbb R^n$.  If $F$ happens to be slowly decreasing, then the point $\eta$ satisfying \eqref{E:slow-decrease1} can in fact be found in $\mathbb R^n$.  This is a consequence of Theorem 5 in \cite{Ehrenpreis1955}, which Ehrenpreis calls a ``minimum modulus theorem."  The reason why slowly decreasing functions are important can be found in the following theorem.

\begin{theorem}[\cite{Ehrenpreis1960}; see also \cite{Ho2}, Ch.~XVI.]\label{T:CT-surjectivity}

Let $T\in\mathscr E'(\mathbb R^n)$. Then the following conditions on $T$ are equivalent:
\begin{enumerate}[(a)]
\item The convolution operator $C_T\colon \mathscr E(\mathbb R^n)\to\mathscr E(\mathbb R^n)$ is surjective.
\item The convolution operator $C_T\colon \mathscr D'(\mathbb R^n)\to\mathscr D'(\mathbb R^n)$ is surjective.
\item There is a distribution $S\in\mathscr D'(\mathbb R^n)$ such that $S*T=\delta_0$.
\item The Fourier-Laplace transform $\widetilde T$ is slowly decreasing.
\item For any $S\in\mathscr E'(\mathbb R^n)$ such that $\widetilde S/\widetilde T$ is an entire function on $\mathbb C^n$, there is a distribution $\mu\in\mathscr E'(\mathbb R^n)$ such that $ \widetilde S=\widetilde T\cdot \widetilde \mu $.
\item If $S\in\mathscr E'(\mathbb R^n)$ and $S*T\in C^\infty(\mathbb R^n)$, then $S\in C^\infty(\mathbb R^n)$
\end{enumerate}
\end{theorem}

Since nonzero polynomials on $\mathbb C^n$ are slowly decreasing, Ehrenpreis' theorem implies, in particular, that any nonzero constant coefficient differential operator $D$ on $\mathbb R^n$ has a fundamental solution, and is surjective on $\mathscr E(\mathbb R^n)$, as well as on $\mathscr D'(\mathbb R^n)$.  More generally, in \cite{Ehrenpreis1955} Ehrenpreis also showed that convolution with any distribution supported in finitely many points in $\mathbb R^n$ is surjective on $\mathscr E(\mathbb R^n)$ (and $\mathscr D'(\mathbb R^n))$.

A distribution $T\in\mathscr E'(\mathbb R^n)$ which satisfies any of the equivalent conditions in Theorem \ref{T:CT-surjectivity} is said to be \emph{invertible}.  

Let us now return to the two-snapshot problem, namely: given functions $f_0$ and $f_1$ in $\mathscr E(\mathbb R^n)$, does there exist a wave $u(x,t)\in C^\infty(\mathbb R^n\times\mathbb R)$ such that $u_0=f_0$ and $u_1=f_1$; moreover, if $u$ exists, is it determined uniquely?

Given the solution \eqref{E:conv-soln1} of the wave equation in terms of its Cauchy data, the question of whether $u$ exists amounts to asking whether there exists a function $g\in\mathscr E(\mathbb R^n)$ such that
\begin{equation}\label{E:initvelocitycond}
g*S_1=f_1-f_0*S_1',
\end{equation}
in which case the formula \eqref{E:conv-soln1} gives us the wave $u$.  In other words, can we impart an initial velocity 
$(\partial_t u)(x,0)=g(x)$ to a wave $u(x,t)$ whose initial position is $u(x,0)=f_0(x)$ so that at time $t=1$ the wave has position (or snapshot) $u(x,1)=f_1(x)$?

To answer this question, we observe that the Fourier-Laplace transform $\widetilde S_1(\zeta)=\sin(\sqrt{\zeta\cdot\zeta})/\sqrt{\zeta\cdot\zeta}$ is a slowly decreasing function on $\mathbb C^n$, since its restriction to $\mathbb R^n$ is $\widetilde S_1(\xi)=\sin (|\xi|)/|\xi|$, a radial sine wave modulated by the distance to the origin.   Theorem \ref{T:CT-surjectivity} therefore implies that the convolution operator $C_{S_1}$ is surjective on $\mathscr E(\mathbb R^n)$, so there does indeed exist a smooth function $g$ satisfying \eqref{E:initvelocitycond}.

The set of all $g\in \mathscr E(\mathbb R^n)$ satisfying \eqref{E:initvelocitycond} is infinite since the kernel of the convolution operator $C_{S_1}$ is infinite-dimensional.  In fact it is easy to see that any exponential of the form $x\mapsto e^{-ix\cdot\zeta}$, where $\widetilde S_1(\zeta)=0$, belongs to the kernel of $C_{S_1}$.  In Theorem \ref{T:C_S1-kernel} below, we provide a more precise description of $\ker C_{S_1}$. (See also \cite{Ehrenpreis1970}, \cite{BerensteinTaylor}, and \cite{StruppaMemoirs} for more general descriptions of kernels of convolutions with invertible distributions.) This will be needed in the sequel. 

\begin{proposition}\label{T:two-snapshot-existence}
Given functions $f_0,\,f_1\in\mathscr E(\mathbb R^n)$, there exist infinitely many waves $u(x,t)$ such that $u_0=f_0$ and $u_1=f_1$.
\end{proposition}

For each $\lambda\in\mathbb C$, let $\mathscr E_\lambda$ be the eigenspace of the Laplace operator $\Delta$ given by
$$
\mathscr E_\lambda=\{f\in\mathscr E(\mathbb R^n)\,\colon\,\Delta f=-\lambda^2\,f\}.
$$
Note that $\mathscr E_\lambda$ is nonzero since it contains the function $(x_1,\ldots,x_n)\mapsto e^{i\lambda x_1}$.

\begin{lemma}\label{T:eigenspace1}
Each $\mathscr E_\lambda$ is an eigenspace of the convolution operator $C_{S_t}$; in fact, we have
\begin{equation}\label{E:conv-eigenspace1}
g*S_t=\frac{\sin(t\lambda)}{\lambda}\,g
\end{equation}
for all $g\in \mathscr E_\lambda$ and all $t\in\mathbb R$.
\end{lemma}
\emph{Remark.} The equality \eqref{E:conv-eigenspace1} is not surprising, since any element of $\mathscr E_\lambda$ is real analytic, and from \eqref{E:wave-fund-soln} the convolution operator $C_{S_t}$ can be written formally as $C_{S_t}=\sum_{k=1}^\infty (t^{2k+1}\Delta^k)/(2k+1)!$.
\begin{proof}  
Fix $g\in\mathscr E_\lambda$. The relation \eqref{E:conv-eigenspace1} is then immediate from the fact that both sides are solutions of the wave equation \eqref{E:wave-eqn1} with Cauchy data $u(x,0)=0,\; u_t(x,0)=g(x)$.
\end{proof}

Lemma \ref{T:eigenspace1} enables us to provide an explicit description below of the kernel of the convolution operator $C_{S_t}\colon\mathscr E(\mathbb R^n)\to\mathscr E(\mathbb R^n)$.  This may in some sense be considered to be a refinement (for $C_{S_t}$) of the general description of convolution kernels given, for example, in \cite{BerensteinTaylor}, Theorem 9.1.

\begin{theorem}\label{T:C_S1-kernel}
For any $t\neq 0$, the kernel of $C_{S_t}$ is the closure in $\mathscr E(\mathbb R^n)$ of the direct sum
\begin{equation}\label{E:kerC_St}
\bigoplus_{j=1}^\infty \mathscr E_{\pi j/t}.
\end{equation}
\end{theorem}
\begin{proof} Note that each $\mathscr E_{\pi j/t}$ is a Frech\'et space, being a closed subspace of the Frech\'et space $\mathscr E(\mathbb R^n)$.

It is clear from Lemma \ref{T:eigenspace1} that each $\mathscr E_{\pi j/t}$ lies in the kernel of $C_{S_t}$, so the closure of $\oplus_{j=1}^\infty \mathscr E_{\pi j/t}$ in $\mathscr E(\mathbb R^n)$ is a subspace of the kernel of $C_{S_t}$.  Now suppose that this closure is not all of $\ker(C_{S_t})$.  Then there is an $f\in\ker(C_{S_t})$ that does not belong to this closure.  By the Hahn-Banach Theorem, there is a distribution $T\in\mathscr E'(\mathbb R^n)$ such that $T(f)\neq 0$ whereas the restriction of $T$ to each $\mathscr E_{\pi j/t}$ is zero.

Now for each $j$, let $V_j$ denote the algebraic variety in $\mathbb C^n$ given by
$$
\sum_{k=1}^n \zeta_k^2=\frac{\pi^2 j^2}{t^2}\qquad (\zeta=(\zeta_1,\ldots,\zeta_n)\in\mathbb C^n).
$$
For any $\zeta\in V_j$, the function $x\mapsto e^{-i\zeta\cdot x}$ belongs to $\mathscr E_{\pi j/t}$, and it follows $\widetilde T(\zeta)=T(e^{-i\zeta\cdot x})=0$. Thus $\widetilde T$ vanishes on the variety $V_j$.  Now the analytic variety in $\mathbb C^n$ given by $\widetilde S_t(\zeta)=0$ is the disjoint union of the $V_j$. Each $V_j$ is an embedded complex submanifold of $\mathbb C^n$, since the complex differential $\partial\widetilde S_t$ is nowhere vanishing on it.  Since $\widetilde T$ vanishes on $V_j$, the ratio $\widetilde T/\widetilde S_t$ is holomorphic at each point of $V_j$.  It follows that $\widetilde T/\widetilde S_t$ is an entire function on $\mathbb C^n$.

Condition (c) of Theorem \ref{T:CT-surjectivity} then implies that there is a distribution $\Psi\in\mathscr E'(\mathbb R^n)$ such that $T=S_t*\Psi$. Noting that $\widecheck S_t=S_t$, we then obtain $T(f)=(S_t*\Psi)(f)=\Psi(\widecheck S_t*f)=\Psi(S_t*f)=0$, a contradiction.  This completes the proof.
\end{proof}

In the last part of the proof above, we used the following notation: for any test function $\varphi$ and any distribution $S$ on $\mathbb R^n$,  $\widecheck\varphi$ denotes the function  $\widecheck\varphi(x)=\varphi(-x)$ and $\widecheck S$ the distribution $\widecheck S(\varphi)=S(\widecheck\varphi)$.

As mentioned earlier, since $\ker(C_{S_1})$ is infinite-dimensional, there are infinitely many solutions to the convolution equation \eqref{E:initvelocitycond}, and hence infinitely many waves $u(x,t)$ with snapshots $u_0=f_0$ and $u_1=f_1$.  Nonetheless, the snapshots $u_m$ at any integer time $m$ is determined uniquely by $u_0$ and $u_1$, as the following theorem shows.

\begin{theorem}[\cite{Wangthesisrohtua}]\label{T:integersnapshot} 
Suppose that $u(x,t)\in C^\infty(\mathbb R^n\times \mathbb R)$ is a solution of the wave equation \eqref{E:wave-eqn1}. Then for any $m\in\mathbb Z$, the snapshot $u_m$ of $u$ is determined completely by $u_0$ and $u_1$. Explicitly, we have
\begin{equation}\label{E:intsnapshot}
u_m=\Psi_m*u_1-\Psi_{m-1}*u_0,
\end{equation}
where $\Psi_m\in\mathscr E'(\mathbb R^n)$ has Fourier transform given by
\begin{equation}\label{E:psi-m}
\widetilde\Psi_m(\xi)=\frac{\sin(m|\xi|)}{\sin |\xi|}=U_{m-1}(\cos |\xi|)\qquad (\xi\in\mathbb R^n).
\end{equation}
Here $U_m$ represents the Chebyshev polynomial of the second kind, with the supplementary definition $U_{-m-1}=-U_{m-1}$ for $m\in\mathbb Z_+$.
\end{theorem}
\begin{proof}
For any $m\in\mathbb Z$ (including $m$ negative), standard trigonometric identities give us
\begin{align*}
\sin((m+2)|\xi|)+\sin (m|\xi|)&=2\cos |\xi|\,\sin((m+1)|\xi|)\\
\cos((m+2)|\xi|)+\cos (m|\xi|)&=2\cos |\xi|\,\cos((m+1)|\xi|).
\end{align*}
Since $S_m,\,S_m'$, and $\Psi_m$ all belong to $\mathscr E'(\mathbb R^n)$, we can take the inverse Fourier transform and use \eqref{E:wave-fund-soln} and \eqref{E:psi-m} to obtain
\begin{align*}
\Psi_{m+2}+\Psi_m&=2\,S_1'*\Psi_{m+1}\\
S_{m+2}+S_m&=2\,S_1'*S_{m+1}\\
S_{m+2}'+S_m'&=2\,S_1'*S_{m+1}'.
\end{align*}
Letting $u_0=f$ and $u_1=g$, \eqref{E:conv-soln1} gives us
\begin{align}
u_{m+2}+u_m&=(S_{m+2}'+S_m')*f+(S_{m+2}+S_m)*g\nonumber\\
&=2\,S_1'*S_{m+1}'*f+2\,S_1'*S_{m+1}*g\nonumber\\
&=2\,S_1'*u_{m+1}.\label{E:u-recurs}
\end{align}
If we put $v_m=\Psi_m*u_1-\Psi_{m-1}*u_0$, we can justify in the same way that $v_m$ satisfies the recursion formula
\begin{equation}\label{E:v-recurs}
v_{m+2}+v_m=2\,S_1'*v_{m+1}.
\end{equation}
By \eqref{E:psi-m} we see that $\Psi_0=0$ and $\Psi_{\pm1}=\pm\delta_0$, and hence that $v_1=u_1$ and $v_0=u_0$. Then \eqref{E:u-recurs} and \eqref{E:v-recurs}, and induction, will prove \eqref{E:intsnapshot}.
\end{proof}

Theorem \ref{T:integersnapshot} easily implies that if we are given the snapshots $u_a$ and $u_b$ at times $t=a$ and $t=b$, respectively, of a wave $u(x,t)$, then we can determine the snapshots $u_{a+m(b-a)}$ of $u$ for all $m\in\mathbb Z$.  In fact, an argument exactly like that in the proof above gives us the following result.
\begin{theorem}\label{T:integersnapshot2}
Suppose that $u(x,t)$ is a wave in which we are given the snapshots $u_a$ and $u_b$, where $a<b$.  If we put $s=b-a$, then
\begin{equation}\label{E:general-integer-time}
u_{a+m(b-a)}=u_b*\Psi_{m,s}-u_a*\Psi_{(m-1),s}, \qquad\qquad m\in\mathbb Z,
\end{equation}
where $\Psi_{m,s}\in\mathscr E'(\mathbb R^n)$ has Fourier transform given by
\begin{equation}\label{E:psi-ms}
\widetilde\Psi_{m,s}(\xi)=\frac{\sin(ms|\xi|)}{\sin (s|\xi|)}=U_{m-1}(\cos s|\xi|),\qquad\qquad \xi\in\mathbb R^n.
\end{equation}
\end{theorem}

\section{Waves Determined by Three Snapshots}

In the preceding section we saw that two given snapshots can never uniquely determine a whole wave $u(x,t)$. A natural question that then arises is this: can a wave $u(x,t)$ be uniquely determined by \emph{three} of its snapshots?  By time-translating and scaling $(x,t)$ we can assume that the snapshots occur at time $t=0,\,1$, and $\alpha$, where $\alpha\neq 0,1$. Thus if $u_0$, $u_1$, and $u_\alpha$ are given, can we determine the wave $u$? 

First, we claim that these three snapshots of $u$ determine the snapshots $u_t$ at the times $t=j_1\alpha+j_2$, for all integers $j_1$ and $j_2$ such that at least one of them is even.  
To see this, recall that from Theorem \ref{T:integersnapshot2} that
any two snapshots $u_a$ and $u_b$ ($a,b\in\mathbb{R}$) determine all snapshots $u_{a+m(b-a)}$ where $m\in\mathbb{Z}$. 
So clearly, the snapshots $u_0$, $u_1$ and $u_\alpha$ determine the snapshots $u_m$ and $u_{m\alpha}$ for all $m\in\mathbb{Z}$. Then the claim follows from the identities $2j_1\alpha+j_2 = j_1\alpha-((-j_2)-j_1\alpha)$ and $j_1\alpha+2j_2 = -j_1\alpha-2(-j_2-j_1\alpha)$.  

When $\alpha$ is an irrational number, the set of all $t$ given above is dense in $\mathbb R$, so the wave $u(x,t)$ is uniquely determined by continuity. On the other hand, if $\alpha=p/q$ is a nonzero rational number written in lowest terms, then the set of all such $t$ equals the set of all integer multiples of $1/q$, so we can only determine $u(x,t)$ for $t$ an integer multiple of $1/q$.  From the remarks preceding the statement of Theorem \ref{T:integersnapshot}, we see that there are infinitely many waves with snapshots at these rational times.  We formalize these observations into the following result.

\begin{proposition}\label{T:snapshot-uniqueness}  
If $\alpha$ is a irrational number, then a wave $u(x,t)$ is uniquely determined by its snapshots $u_0$, $u_1$, and $u_\alpha$.  
If $\alpha$ is a nonzero rational number with $\alpha=p/q$ in lowest terms, then the snapshots $u_t$ can be determined by $u_0,\,u_1$, and $u_\alpha$ only for $t$ an integer multiple of $1/q$.
\end{proposition}

A related and much more interesting question concerns existence, stated below. This is the main topic we will be dealing with in this section.
\begin{question}\label{Q:three-snapshot}
	Given any functions $f_0,\,f_1$, and $f_\alpha$ in $\mathscr E(\mathbb R^n)$, does there exist a smooth wave $u(x,t)$ such that $u_0=f_0,\,u_1=f_1$, and $u_\alpha=f_\alpha$? (We note that Proposition \ref{T:snapshot-uniqueness} shows that such a wave is unique if and only if $\alpha$ is irrational.)
\end{question}

As before we observe that there is no loss of generality in assuming that our three snapshots occur at times $t=0,\,1$, and $\alpha$.

If a wave $u$ exists with snapshots $f_0,\,f_1$, and $f_\alpha$ at times $t=0,1$, and $\alpha$, respectively, then the convolution formula \eqref{E:conv-soln1} shows that its initial velocity $g(x)=(\partial_t u)(x,0) \in\mathscr{E}(\mathbb{R}^n)$ satisfies the following system of convolution equations:
\begin{equation}\label{E:existenceofg}
\left\{
	\begin{aligned}
		g*S_1 &= f_1-f_0*S_1';\\
		g*S_\alpha &= f_\alpha-f_0*S_\alpha'.
	\end{aligned}
\right.
\end{equation}
Conversely, any $g\in\mathscr{E}(\mathbb{R}^n)$ satisfying \eqref{E:existenceofg} is the initial velocity of a wave $u$ with snapshots $f_0,\,f_1,\,f_\alpha$ at times $t=0,1,\,\alpha$, respectively, since by \eqref{E:conv-soln1} we would then have $u(x,t)=(f_0*S_t')(x)+(g*S_t)(x)$.
We summarize this observation as follows.

\begin{lemma}\label{T:equiv_u_g}
	Given functions $f_0,\,f_1$, and $f_\alpha$ in $\mathscr E(\mathbb R^n)$, there is a one-to-one correspondence of the following two objects:
	\begin{itemize}
		\item Waves $u(x,t)$ with snapshots $(f_0,f_1,f_\alpha)$ at times $t=0,\,1,\,\alpha$.
		\item Functions $g\in \mathscr{E}(\mathbb{R}^n)$ satisfying \eqref{E:existenceofg}.
	\end{itemize}
	 
\end{lemma} 

If \eqref{E:existenceofg} holds for some $g\in\mathscr{E}(\mathbb{R}^n)$, then we see immediately that $f_0,\,f_1$, and $f_\alpha$ must satisfy the \emph{compatibility condition}
\begin{equation}\label{E:compatibility}
	(f_1-f_0*S_1')*S_\alpha=(f_\alpha-f_0*S_\alpha')*S_1.
\end{equation}
A natural question to ask is whether this condition is sufficient to guarantee the existence of a wave with snapshots $(f_0, f_1, f_\alpha)$ at $t=0,\,1,\,\alpha$. We can thus more precisely reformulate Question \ref{Q:three-snapshot} in the following way.

\begin{question}\label{Q:three-snapshot_g}
    For any triple $(f_0,\,f_1,\,f_\alpha)$ of functions in $\mathscr E(\mathbb R^n)$ satisfying the compatibility condition \eqref{E:compatibility}, does there exist a function $g\in\mathscr E(\mathbb R^n)$ that satisfies the system of convolution equations \eqref{E:existenceofg}? 
\end{question}

\begin{remark}
Since $S_\alpha*S'_1-S_\alpha'*S_1=S_{\alpha-1}$ and $S_{-\alpha}=-S_\alpha$, condition \eqref{E:compatibility} can be reformulated equivalently as
\begin{equation}\label{E:compatibility-two}
f_0*S_{\alpha-1}+f_1*S_{-\alpha}+f_\alpha*S_1=0.
\end{equation}
Condition \eqref{E:compatibility-two} has the advantage of being equivalent to the compatibility conditions of the type \eqref{E:compatibility} that would arise by assuming instead the existence of the appropriate Cauchy data at times $t=1$ or at $t=\alpha$.  More generally, it can be readily shown that if we fix an ordered triple $(a,b,c)$ of real numbers, then any triple $(f_a,\,f_b,\,f_c)$ of functions in $\mathscr E(\mathbb R^n)$ are snapshots of some wave $u(x,t)$ at times $t=a,\,b,\,c$, only if they satisfy the condition
\begin{equation}\label{E:compatibility-three}
f_a*S_{c-b}+f_b*S_{a-c}+f_c*S_{b-a}=0.
\end{equation}
While it is possible to work with condition \eqref{E:compatibility-two}, our proofs will turn out to be easier if we primarily work with the compatibility condition \eqref{E:compatibility} to determine whether it is sufficient to guarantee the existence of a wave with the given snapshots at times $0,\,1$, and $\alpha$.
\end{remark}

Recall that our discussion around \eqref{E:initvelocitycond} in the last section shows that the convolution operator $C_{S_1}$ is surjective on $\mathscr{E}(\mathbb{R}^n)$. It can be proven in the same way that $C_{S_t}$ is surjective on $\mathscr{E}(\mathbb{R}^n)$ for all $t\neq 0$. Therefore we can always find functions $g_1$ and $g_2$ in $\mathscr{E}(\mathbb{R}^n)$ such that 
\begin{equation*}
	\left\{
	\begin{aligned}
		g_1*S_1 &= f_1-f_0*S_1';\\
		g_2*S_\alpha &= f_\alpha-f_0*S_\alpha'.
	\end{aligned}
	\right.
\end{equation*}
The key question then is whether it is possible to have $g_1$ and $g_2$ be the same function.

Question \ref{Q:three-snapshot_g} can be clarified a bit by abstracting it into a linear algebra question. 
To be precise, suppose that $S$ and $T$ are two commuting and surjective linear operators on a vector space $V$. Given two vectors $v,w\in V$ that satisfy $Sv=Tw$, does there exist a vector $x\in V$ such that $v=Tx$ and $w=Sx$? For convenience, we introduce a bit of terminology (which might be nonstandard) to describe this property. 

\begin{definition}\label{D:joint-pair}
	Let $V$ be a vector space, and $S,\,T$ be two linear operators on $V$. A pair $(v,w)\in V\times V$ is called a \emph{joint pair} for $(T,S)$ in $V$ provided that there is a vector $x\in V$ satisfying $v=Tx$ and $w=Sx$. The set of all joint pairs for $(T,S)$ in $V$ is thus $\{(Tx,Sx)\in V\times V: x\in V\}$.
\end{definition}

If $ST=TS$ in the definition above, then any joint pair $(v,w)$ for $(T,S)$ in $V$ must satisfy the compatibility condition $Sv=Tw$. Does this condition suffice to guarantee that $(v,w)$ is a joint pair as well? The proposition below provides a criterion which is of fundamental importance in the proof of our main results in this section.

\begin{proposition}\label{T:linalg1}
	Let $V$ be a vector space, and $S,\,T$ be two surjective linear operators with $ST=TS$. Then the following statements are equivalent. 
	\begin{enumerate}[(a).] 
		\item $S:\ker T\to\ker T$ is surjective.
		\item $T:\ker S\to\ker S$ is surjective.
		\item $\ker(ST)\subset \ker S +\ker T$. (This is equivalent to $\ker(ST)=\ker S+\ker T$.)
		\item Any pair $(v,w)\in V\times V$ with $Sv=Tw$ is a joint pair of $(T,S)$ in $V$. 
	\end{enumerate}
\end{proposition}

\begin{proof}  The proof is a straightforward exercise in linear algebra, but we include it for completeness.
	We first prove that (a)$\iff $(c). Note that since $S$ and $T$ commute, $S$ maps $\ker T$ to $\ker T$ and $T$ maps $\ker S$ to $\ker S$
	
	For (a)$\implies$(c), for any $x_0$  in $\ker(ST)$, we have $Sx_0\in\ker T$. From $(a)$, we can find some $x_T\in\ker T$ such that $Sx_T=Sx_0$. This gives $x_0 = (x_0-x_T)+x_T\in\ker S+\ker T$.
	
	For (c)$\implies$(a), suppose $y$ belongs to $\ker T$. The surjectivity of $S$ on $V$ gives a vector $x'\in V$ with $Sx'=y$. Then $x'$ belongs to $\ker(ST)$, hence Condition $(c)$ gives a decomposition $x'=x'_S+x'_T$, with $x'_S\in\ker S$ and $x'_T\in\ker T$. We then have $y=Sx'=Sx'_T\in S(\ker T)$.
	
	The same argument shows that (b)$\iff$(c), so (a), (b) and (c) are equivalent.
	
	Let us prove next that (c)$\implies$(d). Since $S$ and $T$ are surjective, there exist vectors $v_1$ and $w_1$ in $V$ such that $v=Tv_1$ and $w=Sw_1$. Then $v_1-w_1\in\ker (ST)\subset\ker S+\ker T$.  This gives $v_1-w_1=u_S+u_T$, where $u_S\in\ker S$ and $u_T\in\ker T$, and the required vector $x$ in $(d)$ can be chosen as
	$$ x=w_1+u_S=v_1-u_T. $$

        Finally, for (d)$\implies$(c), suppose that $v\in\ker(ST)$.  Then $Tv\in\ker S$, so since (d) is assumed to hold, the ordered pair $(Tv,0)$ is a joint pair for $(T,S)$.  Hence $(Tv,0)=(Tx,Sx)$ for some $x\in V$.  It follows that $v=x+(v-x)\in\ker S+\ker T$.
	
	This completes the proof.
\end{proof}

\begin{remark}\label{R:linalgremark}  Retaining the hypotheses of the preceding proposition, consider any joint pair $(v_0,w_0)$ for $(T,S)$.  It is not hard to see that the set of all $w\in V$ such that $Sv_0=Tw$ coincides with the subset $w_0+\ker T\subset V$, and the set of all $w\in V$ such that $(v_0,w)$ is a joint pair coincides with the subset $w_0+S(\ker T)\subset w_0+\ker T\subset V$.

For each $v_0\in V$, the surjectivity of $T$ implies that there is at least one $w_0\in V$ such that $(v_0,w_0)$ is a joint pair.  Thus the set of all $w$ such that $(v_0,w)$ is a joint pair coincides with the set of all $w$ such that $Sv_0=Tw$ if and only if $S(\ker T)=\ker T$.
\end{remark}

Let us now apply Proposition \ref{T:linalg1} and Remark \ref{R:linalgremark} to the specific case of Question \ref{Q:three-snapshot_g}, where $V=\mathscr{E}(\mathbb{R}^n)$, $S=C_{S_\alpha}$, $T=C_{S_1}$, $v=f_1-f_0*S_1'$, and $w=f_\alpha-f_0*S_\alpha'$. 
Then it is clear that the key to answering the question is to determine whether $S_\alpha:\ker S_1\to\ker S_1$ is surjective, or, equivalently, whether $\ker (S_\alpha S_1)$ is a subset of $\ker S_\alpha+\ker S_1$.  

We will see that the answer to this question will depend fundamentally on the nature of the real number $\alpha$, and in particular (when $\alpha$ is irrational) whether the sequence $\{\sin \pi k\alpha\}$ of $k$ decreases slower than some negative power of $k$.  This is precisely a problem of ``small denominators,'' which had its origins in the study of the $n$-body problem more than a century ago, and which was the subject of groundbreaking studies by Arnold, et al.~in the 1960s.  Again, we refer to Yoccoz's paper \cite{Yoccoz1992} for references.

In particular we are led to the notion of  Liouville numbers, so let us recall their definition.

\begin{definition}\label{D:liouville-number}
	A real number $\alpha$ is said to be a \emph{Liouville number} (and we say that $\alpha$ is Liouville), provided that for any positive integer $N$, there is a rational number $p/q$, with $q\geq 2$, such that
	\begin{equation}\label{E:liouville-def}
	0<\left|\alpha-\frac{p}{q}\right|<\frac{1}{q^N}. 
	\end{equation}
\end{definition}

Liouville numbers can therefore be considered ``almost rational'' in that they can be approximated rapidly (in the sense above) by rational numbers.  However, Liouville numbers themselves are irrational, and are in fact transcendental. The set of all Liouville numbers in $\mathbb{R}$ is dense and uncountable, but has Lebesgue measure zero.  

The fraction $p/q$ in \eqref{E:liouville-def} need not be in lowest terms; reducing it to lowest terms one obtains a similar inequality but with a new denominator.

One example of a Liouville number is the sum of any infinite series of the form $\sum_{k=1}^\infty a_k/q^{k!}$, where $q$ is any integer $\geq 2$ and $a_k\in\{1,\ldots, q-1\}$.  It is known that $\pi$ and $e$ are not Liouville numbers.

We will only need the above facts about Liouville numbers, but we refer the reader interested in further details about them to Hardy and Wright's classical text \cite{Hardy-Wright} or J. Steuding's text on Diophantine analysis \cite{Steuding2005}.

Before proving the main results in this section, we offer the following remark about the role of the irrationality of $\alpha$. 
In Proposition \ref{T:linalg1} it is obvious that the following statements are equivalent: 	
\begin{itemize}
    \item $S$ is injective on $\ker T$.
    \item $T$ is injective on $\ker S$.
    \item $\ker S\cap\ker T=\{0\}$.
\end{itemize}
If one of these conditions holds, then the word ``surjective'' in Condition $(a)$ and $(b)$ in Proposition \ref{T:linalg1} can be substituted by ``bijective'', while the sum of $\ker S$ and $\ker T$ in $(c)$ is actually a direct sum.
The lemma below asserts that in our specific case where $S=C_{S_\alpha}$ and $T=C_{S_1}$, the above conditions hold if and only $\alpha$ is an irrational number.

\begin{lemma}\label{T:csalphainjectivity}
	Let $\alpha$ be a real number. The convolution operator $C_{S_\alpha}$ is injective on $\ker(C_{S_1})$ if and only if $\alpha$ is irrational.
\end{lemma}
\begin{proof}  
	Our assertion says that $\ker (C_{S_1})\cap \ker (C_{S_\alpha})=\{0\}$ if and only if $\alpha$ is irrational.  This is obvious when $\alpha=0$, so we can assume that $\alpha\neq 0$.
	
	Assume first that $\alpha$ is rational, with $\alpha=p/q$ in lowest terms, with $q>0$. Then according to Theorem \ref{T:C_S1-kernel}, the nonzero eigenspace $\mathscr E_{q\pi}$ is a subspace of both $\ker (C_{S_1})$ and $\ker (C_{S_\alpha})$.
	
	Next assume that $\alpha$ is irrational, and suppose that $g\in\ker(C_{S_1})\cap \ker(C_{S_\alpha})$.  Consider the wave $u(x,t)$ with Cauchy data $u(x,0)=0,\;(\partial_t u)(x,0)=g(x)$. The convolution solution \eqref{E:conv-soln1} then shows that the snapshots $u_0,\,u_1$, and $u_\alpha$ are identically zero. Since $\alpha$ is irrational, Proposition \ref{T:snapshot-uniqueness} shows that $u(x,t)\equiv 0$, which forces $g=0$.
\end{proof}
We are now in a position to show how the answer to the three-snapshot question \ref{Q:three-snapshot_g} depends on the nature of our real number $\alpha$.

\subsection{The Case When $\alpha$ is a Liouville number}
\ \\ \newline
When $\alpha$ is a Liouville number, the following theorem provides a negative answer to  Question \ref{Q:three-snapshot_g}.

\begin{theorem}\label{T:liouville2}
	Let $\alpha$ be a Liouville number. Given any two functions $f_0$ and $f_1$ in $\mathscr E(\mathbb R^n)$, there exists a function $f_\alpha$ in $\mathscr E(\mathbb R^n)$ satisfying the compatibility condition \eqref{E:compatibility}: 
	$$ (f_1-f_0*S_1')*S_\alpha=(f_\alpha-f_0*S_\alpha')*S_1, $$
	but such that the triple $(f_0,f_1,f_\alpha)$ cannot be the snapshot data of a smooth wave $u(x,t)$ at times $t=0,\,1,\,\alpha$.
\end{theorem}


In view of Remark \ref{R:linalgremark} and the succeeding discussion, the key to proving Theorem \ref{T:liouville2} is that the range $C_{S_\alpha}(\ker(C_{S_1}))\subsetneq\ker(C_{S_1})$, as the lemma below shows. 
In the proof of the lemma, a crucial tool is the description of $\ker(C_{S_1})$ given by Theorem \ref{T:C_S1-kernel}.

\begin{lemma}\label{T:liouville1}
	Let $\alpha$ be a Liouville number. Then the convolution operator  $C_{S_\alpha}\colon \ker(C_{S_1})\to\ker(C_{S_1})$ is not surjective. 
\end{lemma}

\begin{proof}
	The Liouville number $\alpha$ is necessarily irrational, so by Lemma \ref{T:csalphainjectivity}, the linear operator $C_{S_\alpha}\colon \ker(C_{S_1})\to\ker(C_{S_1})$ is injective.  
	Now suppose to the contrary that it is also surjective. Note that $\ker(C_{S_1})$ is a Frech\'et space, being a closed subspace of the Frech\'et space $\mathscr{E}(\mathbb R^n)$. The Open Mapping Theorem then implies that the inverse linear map $C_{S_\alpha}^{-1}\colon \ker(C_{S_1})\to\ker(C_{S_1})$ is continuous.

	Now according to Theorem \ref{T:C_S1-kernel}, the eigenspaces $\mathscr E_{m\pi}$ is contained in $\ker(C_{S_1})$ for all integers $m$, and the convolution operator $C_{S_\alpha}$ on $\mathscr E_{m\pi}$ is multiplication by $\sin(\pi m\alpha)/(m\pi)$. Hence the restriction of 
	$C_{S_\alpha}^{-1}$ to $\mathscr E_{m\pi}$ is scalar multiplication by $(m\pi)/\sin(\pi m\alpha)$.
	The main idea we will be using is that since the Liouville number $\alpha$ can be rapidly approximated by a sequence $\{p_k/q_k\}$ of rational numbers, the sequence $\{\sin(\pi q_k \alpha)\}$ will go to zero faster than any fixed power of $q_k$, as $k\to\infty$. 
	To be precise, since $\alpha$ is a Liouville number, given any positive integer $k$, there is a fraction $p_k/q_k$, with $q_k\geq 2$, such that
\begin{equation}\label{E:integer-estimate}
	|q_k\alpha-p_k|<\frac{1}{q_k^{k-1}}.
\end{equation}
	It is clear that $q_k\to\infty$ as $k\to\infty$, although this does not matter for what follows.
For each $k$ this gives
$$
	0<|\sin (\pi q_k\alpha)|=|\sin(\pi(q_k\alpha-p_k))|\leq \pi\,|q_k\alpha-p_k|<\frac{\pi}{q_k^{k-1}}.
$$
	Hence
\begin{equation}\label{E:sin-estimate}
	\left|(\sin (\pi q_k\alpha))^{-1}\right|>\frac{q_k^{k-1}}{\pi}.
\end{equation}

	Now consider the sequence $\{f_k\}$ in $\mathscr E(\mathbb R^n)$, with
$$
	f_k(x)=\frac{e^{i\pi q_k x_1}}{q_k^k}, \qquad x=(x_1,\ldots,x_n)\in\mathbb R^n,\;k=1,2,\ldots.
$$
	Clearly, $f_k\in\mathscr{E}_{q_k\pi}\subset\ker(C_{S_1})$. We claim that $f_k\to 0$ in $\mathscr E(\mathbb R^n)$, and hence also in the subspace $\ker(C_{S_1})$. In fact, for any constant coefficient differential operator $D$ on $\mathbb R^n$, there is a polynomial $P$ in one variable (independent of $k$) such that
$$
	D f_k=P(q_k)f_k=\frac{P(q_k)}{q_k^k}\,e^{i\pi q_k x_1}.
$$
	This shows that $\lim_{k\to\infty} Df_k(x)=0$, (in fact) uniformly for all $x\in\mathbb R^n$. 
	
	On the other hand, the inequality \eqref{E:sin-estimate} implies that for any $x\in\mathbb R^n$, 
$$
	\left|\left(C_{S_\alpha}^{-1}(f_k)\right)(x)\right| 
	= \left| \frac{q_k\pi}{\sin(\pi q_k\alpha)}\cdot\frac{e^{i\pi q_k x_1}}{q_k^k} \right|
	= \frac{\pi|\sin(\pi q_k\alpha)|^{-1}}{q_k^{k-1}}>1.
$$
	Thus, $C_{S_\alpha}^{-1}(f_k)$ does not even converge pointwise to $0$, which certainly contradicts the continuity of $C^{-1}_{S_\alpha}$ on $\mathscr E(\mathbb R^n)$. This proves the theorem. 
\end{proof}

\begin{remark}
   Since $C_{S_\alpha}$ is not surjective on $\ker C_{S_1}$, the complementary set $\ker(C_{S_1})\setminus C_{S_\alpha}(\ker(C_{S_1}))$ is uncountable.
   Even more, according to the Open Mapping Theorem (see e.g.\ \cite[Theorem 2.11]{RudinFA}), this non-surjectivity implies that $C_{S_\alpha}(\ker(C_{S_1}))$ is a set of first category in the Frech\'et space $\ker(C_{S_1})$. Consequently, ``most'' elements of $\ker(C_{S_1})$ do not belong to the range $C_{S_\alpha}(\ker(C_{S_1}))$.
\end{remark}

\begin{remark} 
Let $\{f_k\}$ be the sequence in the proof above.  Then it is not hard to see that the series $\sum_{k=1}^\infty f_k$ converges in $\mathscr E(\mathbb R^n)$ to an element of the set $\ker(C_{S_1})\setminus C_{S_\alpha}(\ker(C_{S_1}))$. 
\end{remark}

%

Now Theorem \ref{T:liouville2} becomes a straightforward consequence of Lemma \ref{T:liouville1}.

\begin{proof}[Proof of Theorem \ref{T:liouville2}]
Fix a wave $U(x,t)$ with snapshot data $U_0=f_0$ and $U_1=f_1$.  (There are infinitely many such waves, according to Proposition \ref{T:two-snapshot-existence}.)  Now we apply Proposition \ref{T:linalg1} (and the succeeding Remark \ref{R:linalgremark}) with $V=\mathscr E(\mathbb R^n)$, $T=C_{S_1}$, $S=C_{S_\alpha}$, $v_0=f_1-f_0*S_1'$, and $w_0=U_\alpha-f_0*S_\alpha'$.
Then $(f_1-f_0*S_1',\,U_\alpha-f_0*S_\alpha')$ is a joint pair for $(C_{S_1},C_{S_\alpha})$.

From Remark \ref{R:linalgremark} we see that the set of all $f_\alpha\in\mathscr E(\mathbb R^n)$ which satisfy the compatibility condition \eqref{E:compatibility} is the subset $U_\alpha+\ker C_{S_1}\subset \mathscr E(\mathbb R^n)$, whereas the set of all $f_\alpha$ for which there is a wave $u(x,t)$ such that $u_0=f_0,\,u_1=f_1$, and $u_\alpha=f_\alpha$ coincides with the set $U_\alpha+C_{S_\alpha}(\ker C_{S_1})$.

The last part of the proof of Lemma \ref{T:liouville1} shows that $C_{S_\alpha}(\ker C_{S_1})$ is of first category in $\ker C_{S_1}$.  Therefore for  ``most'' $f_\alpha\in\mathscr E(\mathbb R^n)$ such that $f_0,\,f_1$ and $f_\alpha$ satisfy the compatibility condition \eqref{E:compatibility}, there is no wave $u(x,t)$ whose snapshots at time $t=0,\,1$, and $\alpha$ are $f_0,\,f_1$, and $f_\alpha$, respectively.

This completes the proof of Theorem \ref{T:liouville2}.
\end{proof}

\subsection{The Case When $\alpha$ Is Irrational and Not Liouville}
\ \\ \newline
In this subsection, we assume that $\alpha$ is an irrational number and \emph{not} a Liouville number. As noted earlier, these numbers $\alpha$ constitute almost all real numbers (with respect to Lebesgue measure). The main result is the following.

\begin{theorem}\label{T:irrational-not-Liouville}
	Let $\alpha$ be irrational and not a Liouville number. Suppose $f_0,\,f_1$ and $f_\alpha\in\mathscr E(\mathbb R^n)$ satisfy the compatibility condition \eqref{E:compatibility}:
	$$ (f_1-f_0*S_1')*S_\alpha=(f_\alpha-f_0*S_\alpha')*S_1. $$
	Then there exists a unique smooth wave $u(x,t)$ whose snapshots at times $t=0,1,\alpha$ are precisely the given triple $(f_0,f_1,f_\alpha)$.
\end{theorem}

In order to prove the theorem, and in view of Proposition \ref{T:linalg1} it is worthwhile to consider the following more general question. 
Let $\mu$ and $\nu$ be two distributions in $\mathscr E'(\mathbb R^n)$. Suppose $f$ and $h$ are two functions in $\mathscr E(\mathbb R^n)$ such that $f*\mu=h*\nu$, under what conditions is $(f,h)$ a joint pair of $(C_\nu,C_\mu)$?
It turns out that the answer is affirmative under a certain joint decay condition on the Fourier transforms $\widetilde\mu$ and $\widetilde\nu$, which was first obtained by Ehrenpreis in his 1970 book \cite{Ehrenpreis1970}. 
The crucial result is as follows.  (Its proof can be found in the book, in the remarks after Theorem 11.2.) 

\begin{theorem}\label{T:Ehrenpreis-estimate} (Ehrenpreis, 1970.)
	Let $\mu,\,\nu\in\mathscr E'(\mathbb R^n)$, and suppose that there is a constant $A>0$ such that
	\begin{equation}\label{E:exp-estimate1}
		|\widetilde\mu(\zeta)|+|\widetilde\nu(\zeta)|\geq (A+|\zeta|)^{-A}\,\exp(-A\,|\mathrm{Im}(\zeta)|),\qquad\text{for all }\zeta\in\mathbb C^n.
	\end{equation}
	Then $1$ belongs in the ideal of the ring $\widetilde{\mathscr E}'(\mathbb R^n)$ generated by $\widetilde\mu$ and $\widetilde\nu$.  Thus the ideal in the convolution ring $\mathscr E'(\mathbb R^n)$ generated by $\mu$ and $\nu$ is $\mathscr{E}'(\mathbb R^n)$ itself.  In particular, there exist distributions $\Phi$ and $\Psi$ in $\mathscr E'(\mathbb R^n)$ such that
	\begin{equation}\label{E:unit-ideal}
		\Phi*\mu+\Psi*\nu=\delta_0,
	\end{equation}
	where $\delta_0$ is the Dirac delta function.
\end{theorem}

Ehrenpreis' result above implies the following result about joint pairs under convolutions.

\begin{lemma}\label{T:kernelsum}
	Let $\mu,\,\nu\in\mathscr E'(\mathbb R^n)$, and assume that there is a constant $A>0$ such that the decay estimate \eqref{E:exp-estimate1} holds. Suppose that $f$ and $h$ are functions in $\mathscr E(\mathbb R^n)$ such that 
	$$ f*\mu=h*\nu. $$
	Then $(f,h)$ is a joint pair of $(C_\nu,C_\mu)$ in $\mathscr E(\mathbb R^n)$, namely there is a function $g\in\mathscr E(\mathbb R^n)$ such that
	\begin{equation}\label{E:conv-joint-pair}
	 f=g*\nu,\quad h=g*\mu.
	\end{equation}
\end{lemma}

\begin{proof}
	Since the Fourier transforms $\widetilde\mu$ and $\widetilde\nu$ satisfy the decay estimate \eqref{E:exp-estimate1}, Theorem \ref{T:Ehrenpreis-estimate} implies that there are distributions
	$\Phi$ and $\Psi$ in $\mathscr E'(\mathbb R^n)$ such that the relation \eqref{E:unit-ideal} holds.  Let $g=h*\Phi+f*\Psi$.  Then the hypothesis $f*\mu=h*\nu$ easily implies \eqref{E:conv-joint-pair}.
\end{proof}

In the present case in which $\alpha$ is irrational and not Liouville, the next theorem shows that the Fourier transforms $\widetilde S_1$ and $\widetilde S_\alpha$ satisfy a joint decay estimate even stronger than \eqref{E:exp-estimate1}. The idea is that since $\alpha$ is not Liouville, $\sin x$ and $\sin\alpha x$ will not be ``too close'' to 0 simultaneously for all $x\in\mathbb R$.

\begin{theorem}\label{T:fund-estimate}
Suppose that $\alpha$ is not a Liouville number and is irrational.  Then there exists a constant $C>0$ and an integer $N>0$ such that
\begin{equation}\label{E:fund-estimate1}
|\widetilde S_1(\zeta)|+|\widetilde S_\alpha(\zeta)|\geq C\,(1+|\zeta|)^{-N},
\end{equation}
for all $\zeta\in \mathbb C^n$.
\end{theorem}
\begin{proof}
We will first show that there exists a positive constant $C$ and a positive integer $N$ such that
\begin{equation}\label{E:fund-estimate2}
\frac{|\sin x|}{|x|}+\frac{|\sin (\alpha x)|}{|x|} \geq C\,(1+|x|)^{-N}
\end{equation} 
for all $x\in\mathbb R$.  Now since $\alpha$ is irrational, the left hand side above is never zero, so it has a positive minimum
value on any compact interval. Moreover, the functions $(\sin x)/x$ and $(\sin (\alpha x))/x$ are even.  It is therefore
enough for us to prove the estimate \eqref{E:fund-estimate2} (for appropriate $C$ and $N$) when $x> 2\pi$.  The inequality \eqref{E:fund-estimate2} will then follow for all $x$ by adjusting $C$ if necessary.

Since we are assuming that $\alpha$ is irrational and not a Liouville number, there is a positive integer $N$ such that
\begin{equation}\label{E:not-Liouville}
\left|\alpha-\frac pq\right|\geq \frac{1}{q^N}
\end{equation}
for all integers $p$ and all integers $q\geq 2$.

Let $q$ be an integer closest to $x/\pi$ and $p$ an integer closest to $\alpha x/\pi$.   
Since $x>2\pi$, we have $q\geq 2$; moreover
\begin{equation}\label{E:closest-int}
\left|\frac{x}{\pi}-q\right|\leq\frac{1}{2}\quad \text{and}\quad \left|\frac{x}{\pi}-\frac{p}{\alpha}\right|\leq\frac{1}{2|\alpha|}.
\end{equation}
Now \eqref{E:not-Liouville} implies that
$$
	\left|q-\frac p\alpha\right|\geq \frac{1}{|\alpha|\, q^{N-1}},
$$
so it follows that
$$
	\text{either}\quad\left|\frac{x}{\pi}-q\right|\geq\frac{1}{2|\alpha|\,q^{N-1}}\quad\text{or}\quad \left|\frac{x}{\pi}-\frac{p}{\alpha}\right|\geq \frac{1}{2|\alpha|\,q^{N-1}}.
$$
Let $D=\min\{\pi/2,\,\pi/(2|\alpha|)\}$. Then we see that
$$
	\text{either}\quad|x-q\pi|\geq\frac{D}{q^{N-1}}\quad\text{or}\quad |\alpha x-p\pi|\geq \frac{D}{q^{N-1}}.
$$
Now \eqref{E:closest-int} also shows that $|x-q\pi|\leq\pi/2$ and $|\alpha x-p\pi|\leq \pi/2$. Noting that $|(\sin t)/t|\geq2/\pi\geq 1/2$ for all $t\in [-\pi/2, \pi/2]$, we obtain
\begin{align*}
	|\sin x|+|\sin (\alpha x)|&=|\sin (x-q\pi)|+|\sin (\alpha x-p\pi)|\\
	&\geq\frac 12\,|x-q\pi|+\frac 12\,|\alpha x-p\pi|\\
	&\geq \frac{D}{2q^{N-1}}\\
	&\geq \frac{D}{2(1/2+(x/\pi))^{N-1}}\\
	&>\frac{D}{2(1+x)^{N-1}}.
\end{align*}
Letting $C=D/2$, the inequality \eqref{E:fund-estimate2} follows for $x>2\pi$. Hence, as mentioned earlier, modulo a possible adjustment of $C$, \eqref{E:fund-estimate2} holds for all $x\in\mathbb R$.

Next let us prove that for the same integer $N$ as in \eqref{E:fund-estimate2} there is a positive constant $C_1$ such that the inequality 
\begin{equation}\label{E:fund-estimate4}
\left|\frac{\sin z}{z}\right|+\left|\frac{\sin (\alpha z)}{z}\right|\geq C_1\,(1+|z|)^{-N}
\end{equation}
holds for all $z\in\mathbb C$.
  
For this, put $z=x+iy$ where $x,y\in\mathbb{R}$, so that $|\sin z|^2=\sin^2 x+\sinh^2 y$, and $|\sin (\alpha z)|^2=\sin^2 (\alpha x)+\sinh^2 (\alpha y)$.  Then by \eqref{E:fund-estimate2} we obtain
\begin{align*}
	|\sin z|+|\sin(\alpha z)|
	&\geq\frac{1}{2}\,\left(|\sin x|+|\sin (\alpha x)|\right)+\frac{1}{2}\,\left(|\sinh y|+|\sinh (\alpha y)|\right)\\
	&\geq  \frac C2\,|x|\,(1+|x|)^{-N}+\frac{(1+|\alpha|)\,|y|}{2}\\
	&\geq C_1\,(|x|+|y|)\,(1+|x|)^{-N}\\
	&\geq C_1\,|z|\,(1+|z|)^{-N},
\end{align*}
where we have put $C_1=\min\{C/2,\,(1+|\alpha|)/2\}$.  This proves \eqref{E:fund-estimate4}.

Finally, suppose that $\zeta=(\zeta_1,\ldots,\zeta_n)\in\mathbb C^n$.  We choose a $z\in \mathbb C$ such that $z^2=\zeta\cdot\zeta$.  It is clear from this that $|z|\leq |\zeta|$. The definition \eqref{E:wave-fund-soln} gives
$\widetilde S_1(\zeta)=(\sin z)/z$ and $\widetilde S_\alpha(\zeta)=(\sin(\alpha z))/z$, so we obtain from \eqref{E:fund-estimate4} that
\begin{align*}
|\widetilde S_1(\zeta)|+|\widetilde S_\alpha(\zeta)|&=\left|\frac{\sin z}{z}\right|+\left|\frac{\sin (\alpha z)}{z}\right|\\
&\geq C_1\,(1+|z|)^{-N}\\
&\geq C_1\,(1+|\zeta|)^{-N}.
\end{align*}
This is \eqref{E:fund-estimate1}, if we rename $C_1$ as $C$.
\end{proof}
Our main result Theorem \ref{T:irrational-not-Liouville} is now an immediate consequence of the above result.

\begin{proof}[Proof of Theorem \ref{T:irrational-not-Liouville}]
The preceding lemma says that the Fourier transforms $S_1$ and $S_\alpha$  satisfy the joint decay estimate \eqref{E:fund-estimate1}, and hence satisfies \eqref{E:exp-estimate1} as well, for some $A>0$. Therefore we can invoke Lemma \ref{T:kernelsum} with $\mu=S_\alpha$, $\nu=S_1$, $f=f_1-f_0*S_1'$, and $h=f_\alpha-f_0*S_\alpha'$.  This gives us a function $g\in\mathscr{E}(\mathbb{R}^n)$ satisfying \eqref{E:conv-joint-pair}, which in our case becomes $f_1-f_0*S_1'=g*S_1,\,f_\alpha-f_0*S_\alpha'=g*S_\alpha$.  

The relations above show that $g$ is an initial velocity for a wave with snapshots $f_0,\,f_1$, and $f_\alpha$ at $t=0,1$, and $\alpha$, respectively.  Since $\alpha$ is irrational, Proposition \ref{T:snapshot-uniqueness} shows that this wave (as well as $g$) is unique.
\end{proof}

\subsection{The Case When $\alpha$ Is A Rational Number}
\ \\ \newline
Finally we consider the case where $\alpha$ is a rational number not equal to $0$ or $1$.  (Recall that rational numbers are not Liouville numbers.) Write $\alpha=p/q$ in lowest terms, with $q>0$. It turns out (somewhat surprisingly) that the compatibility condition \eqref{E:compatibility} is no longer sufficient to ensure the existence of a wave $u$ with given snapshots $(f_0,\,f_1,f_\alpha)$ at times $t=0,\,1$, and $\alpha$, respectively. Instead, these functions will need to satisfy a stronger compatibility condition, which we formulate below.

Without losing any generality, let us rescale coordinates by $(x,t)\mapsto (qx,qt)$ to consider instead the three-snapshot problem at the times $t=0,\,p$, and $q$.  Explicitly, suppose that $u(x,t)$ is a wave with Cauchy data $u(x,0)=f(x)$ and $(\partial_t u)(x,0)=g(x)$, for $f,\,g\in \mathscr E(\mathbb R^n)$.  Then the snapshots $u_p$ and $u_q$ satisfy
\begin{equation}\label{E:intconvsoln}
\begin{aligned}
u_p-u_0*S_p'&=g*S_p;\\
u_q-u_0*S_q'&=g*S_q.
\end{aligned}
\end{equation}
For any integer $m$, let $\Psi_m$ be the (compactly supported) distribution given in \eqref{E:psi-m}.  By taking Fourier transforms, it is clear that $S_p*\Psi_q=S_q*\Psi_p$.  Therefore the relations \eqref{E:intconvsoln} imply that
\begin{equation}\label{E:integercompatibility}
(u_p-u_0*S_p')*\Psi_q=(u_q-u_0*S_q')*\Psi_p.
\end{equation}
The above compatibility condition is therefore satisfied by any wave $u$ with integer time snapshots $u_0,\,u_p$, and $u_q$.

The condition \eqref{E:integercompatibility} easily implies the original compatibility condition \eqref{E:compatibility}. In fact, we simply note that $\Psi_p*S_1=S_p$ and $\Psi_q*S_1=S_q$, so convolving both sides of \eqref{E:integercompatibility} with $S_1$ gives us the relation
\begin{equation}\label{E:integercompatibility2}
(u_p-u_0*S_p')*S_q=(u_q-u_0*S_q')*S_p,
\end{equation}
which is \eqref{E:compatibility}, suitably modified.

On the other hand, \eqref{E:integercompatibility2} does not imply  \eqref{E:integercompatibility}. To see this, consider the function  $h(x)=e^{i\pi x_1}$.   
It is clear from direct observation (or from Theorem \ref{T:C_S1-kernel}) that $h*S_p=h*S_q=0$.
On the other hand we claim that $h*\Psi_p=p\cdot h$ and $h*\Psi_q=q\cdot h$. This claim simply follows from the observation that
$$
h*\Psi_p(x)=\widetilde\Psi_p(\pi e_1)\,h(x)=p\cdot h(x),
$$
and similarly for $h*\Psi_q$.  Now suppose that the triple $(u_0,u_p,u_q)$ satisfy the compatibility condition \eqref{E:integercompatibility}.  Then the triple $(u_0,u_p+h,u_q+h)$ no longer satisfy \eqref{E:integercompatibility}, whereas they still satisfy \eqref{E:integercompatibility2}.  

We have shown that condition \eqref{E:integercompatibility} is strictly stronger than condition \eqref{E:integercompatibility2}.
This means of course that the original condition \eqref{E:integercompatibility2} is no longer sufficient to guarantee the existence of a wave with given snapshot data at $t=0,\,p,\,q$. 
On the other hand, it turns out that
\eqref{E:integercompatibility} is sufficient to guarantee the existence of such a wave.  

\begin{theorem}\label{T:psi-compatibility}
Suppose that $f_0,\,f_p,\,f_q\in \mathscr E(\mathbb R^n)$ satisfy the condition
\begin{equation}\label{E:integercompatibility3}
(f_p-f_0*S_p')*\Psi_q=(f_q-f_0*S_q')*\Psi_p.
\end{equation}
Then there exists a wave $u$ such that $u_0=f_0,\,u_p=f_p$, and $u_q=f_q$.
\end{theorem}
\begin{proof}
Since $p$ and $q$ are assumed relatively prime, there exist integers $k$ and $l$ such that $kp+lq=1$.  Our starting point is the resulting trigonometric relation
$$
\frac{\sin (kp\sqrt{\zeta\cdot\zeta})}{\sin \sqrt{\zeta\cdot\zeta}}\cos (lq\sqrt{\zeta\cdot\zeta})+\frac{\sin(lq\sqrt{\zeta\cdot\zeta})}{\sin \sqrt{\zeta\cdot\zeta}}\cos(kp\sqrt{\zeta\cdot\zeta})=1,\qquad (\zeta\in\mathbb C^n).
$$
Note that the fractions on the left hand side above are holomorphic functions of $\zeta$, and in fact they are the Fourier transforms of $\Psi_{kp}$ and $\Psi_{lq}$, respectively.  Taking inverse Fourier transforms, we obtain
$$
\Psi_{kp}*S_{lq}'+\Psi_{lq}*S_{kp}'=\delta_0.
$$
Now define the distributions $A$ and $B$ in $\mathscr E'(\mathbb R^n)$ by 
$$
\widetilde A(\zeta)=\frac{\sin (kp\sqrt{\zeta\cdot\zeta})}{\sin (p\sqrt{\zeta\cdot\zeta})}\quad\text{and}\quad \widetilde B(\zeta)=\frac{\sin(lq\sqrt{\zeta\cdot\zeta})}{\sin(q\sqrt{\zeta\cdot\zeta})}
$$
It is clear that $\Psi_{kp}=\Psi_p*A$ and $\Psi_{lq}=\Psi_q*B$. Hence
\begin{equation}\label{E:trig0}
\Psi_p*A*S_{lq}'+\Psi_q*B*S_{kp}'=\delta_0.
\end{equation}
Convolving both sides above with any function $g\in\mathscr E(\mathbb R^n)$ gives
\begin{equation}\label{E:trig1}
(g*\Psi_p)*A*S_{lq}'+(g*\Psi_q)*B*S_{kp}'=g.
\end{equation} 
Since $\Psi_p*S_1=S_p$ and $\Psi_q*S_1=S_q$, convolving both sides above yet again with $S_1$ then results in
$$
(g*S_p)*A*S_{lq}'+(g*S_q)*B*S_{kp}'=g*S_1.
$$
Now if there \emph{were} a wave $u(x,t)$ such that $u_0=f_0,\,u_p=f_p$, and $u_q=f_q$, and $g$ was its initial velocity (i.e., $g(x,0)=(\partial_t u(x,0)$), the above equation would give us
\begin{equation}\label{E:g-def1}
(f_p-f_0*S_p')*A*S_{lq}'+(f_q-f_0*S_q')*B*S_{kp}'=g*S_1.
\end{equation}

The idea then is to start with the left hand side of \eqref{E:g-def1}, which we note depends only on $(f_0,f_p,f_q)$, in order to obtain our wave.  Explicitly, consider any function $g\in\mathscr E(\mathbb R^n)$ which satisfies \eqref{E:g-def1}.  (Such a $g$ exists because $S_1$ is invertible.)  Let $u(x,t)$ be the wave with Cauchy data $u(x,0)=f_0(x),\,\partial_t u(x,0)=g(x)$.  We want to prove that $u_p=f_p$ and $u_q=f_q$, which would of course prove the theorem.

This turns out to be a straightforward calculation.  In fact, by \eqref{E:g-def1},
\begin{align*}
g*S_p&=(g*S_1)*\Psi_p\\
&=(f_p-f_0*S_p')*\Psi_p*A*S_{lq}'+(f_q-f_0*S_q')*\Psi_p*B*S_{kp}'
\end{align*}
The hypothesis \eqref{E:integercompatibility3} then shows that
\begin{align*}
g*S_p&=(f_p-f_0*S_p')*\Psi_p*A*S_{lq}'+(f_p-f_0*S_p')*\Psi_q*B*S_{kp}'\\
&=(f_p-f_0*S_p')*(\Psi_p*A*S_{lq}'+\Psi_q*B*S_{kp}')\\
&=f_p-f_0*S_p',
\end{align*}
with the last equality coming from \eqref{E:trig0}. The relation above then gives
$$
u_p-u_0*S_p=f_p-f_0*S_p,
$$ 
which, given that $u_0=f_0$, results in $u_p=f_p$.

An entirely similar calculation  shows that $u_q=f_q$. This finishes the proof.
\end{proof}

\begin{remark}\hfill
\begin{enumerate}[(a)]
\item In the proof above we note that the function $g$ satisfying \eqref{E:g-def1} is determined only up to $\ker(C_{S_1})$.  Thus there are infinitely many waves $u$ having snapshot data $(f_0,f_p,f_q)$ at $t=0,\,p,\,q$.  This is not so surprising since $u$ is determined only up to its snapshots at times $kp+lq$ for all integers $k$ and $l$; i.e., snapshots at all integer times.

\item Another natural way to obtain condition \eqref{E:integercompatibility3} (in a different form) relies on formula \eqref{E:intsnapshot} which provides integer time snapshots. Suppose that the triple $(f_0,f_p,f_q)$ in $\mathscr{E}(\mathscr{\mathbb R}^n)$ are snapshots of a wave $u(x,t)$ at $t=0,\,p,\,q$, we deduce from \eqref{E:intsnapshot} that 
\begin{equation*}
    \begin{aligned}
    f_1*\Psi_p &= f_p+f_0*\Psi_{p-1}; \\
    f_1*\Psi_q &= f_q+f_0*\Psi_{q-1}.
    \end{aligned}
\end{equation*}
These relations imply that
\begin{equation}\label{E:integercompatibility3'}
    (f_p+f_0*\Psi_{p-1})*\Psi_q = (f_q+f_0*\Psi_{q-1})*\Psi_p.
\end{equation}
It is not hard to prove that \eqref{E:integercompatibility3} and \eqref{E:integercompatibility3'} are equivalent, while the key point is the identity
$$ 
(\Psi_{p-1}+S'_p)*\Psi_q = (\Psi_{q-1}+S'_q)*\Psi_p,
$$
which can be shown by taking the Fourier transform and using standard trigonometric identities. Consequently, Condition \eqref{E:integercompatibility3} is also a sufficient and necessary condition for $(f_0,f_p,f_q)$ being snapshots of a wave $u(x,t)$ at $t=0,\,p,\,q$. 

\item We note that just as with \eqref{E:compatibility-three}, we can reformulate the compatibility condition \eqref{E:integercompatibility3} equivalently as
$$
f_0*\Psi_{q-p} + f_q*\Psi_p + f_p*\Psi_{-q}=0.
$$
Convolving both sides above with $S_1$ yields \eqref{E:compatibility-three}.

\end{enumerate}
\end{remark}

\subsection{Support Considerations.} 
\ \\ \newline
Assume that $\alpha$ is irrational.  Regardless of whether $\alpha$ is a Liouville number. it can be shown directly that the compatibility condition \eqref{E:compatibility} is sufficient for the existence of a wave with snapshots $f_0,\,f_1,\,f_\alpha$ at times $t=0,\,1,\,\alpha$, respectively,  as long as these functions are compactly supported.

\begin{theorem}\label{T:snapshot-compact-support}
    Suppose that $\alpha$ is irrational, and $f_0,\,f_1,\,f_\alpha\in\mathscr D(\mathbb R^n)$ satisfy the compatibility condition \eqref{E:compatibility}:
    $$ 
    (f_1-f_0*S_1')*S_\alpha=(f_\alpha-f_0*S_\alpha')*S_1, 
    $$
    Then there exists a unique wave $u(x,t)$ such that $u_t\in\mathscr D(\mathbb R^n)$ for all $t\in\mathbb R$ and such that $u_0=f_0,\,u_1=f_1$, and $u_\alpha=f_\alpha$.
\end{theorem}

\begin{proof}
Since $\alpha$ is irrational, uniqueness follows from Proposition \ref{T:snapshot-uniqueness}. For existence, we note that the Fourier transform of both sides of \eqref{E:compatibility} are well-defined.  If we put $\varphi_1=f_1-f_0*S_1'$ and $\varphi_\alpha=f_\alpha-f_0*S'_\alpha$, we obtain
$$
\widetilde\varphi_1\,\widetilde S_\alpha=\widetilde\varphi_\alpha\,\widetilde S_1.
$$
Since $\widetilde S_1$ and $\widetilde S_\alpha$ have no common zeros, the above equality shows that $\widetilde\varphi_1$ (respectively $\widetilde \varphi_\alpha$) vanishes on the zero set of $\widetilde S_1$ (respectively $\widetilde S_\alpha$).  An argument similar to that in the proof of Theorem \ref{T:C_S1-kernel} then shows that $\widetilde\varphi_1/\widetilde S_1 = \widetilde\varphi_\alpha/\widetilde S_\alpha$ is a holomorphic function on $\mathbb C^n$.

From Theorem \ref{T:CT-surjectivity}(e), we deduce that there is a distribution $T\in\mathscr E'(\mathbb R^n)$ such that
\begin{equation}\label{E:convol-soln}
\varphi_1=T*S_1,\qquad \varphi_\alpha=T*S_\alpha.
\end{equation}
Since $S_1$ is invertible, Theorem \ref{T:CT-surjectivity}(f) shows that in fact $T\in \mathscr D(\mathbb R^n)$.  Renaming $T$ as $g$, \eqref{E:convol-soln} now reads as
$$
\varphi_1=g*S_1,\qquad \varphi_\alpha=g*S_\alpha.
$$
If $u(x,t)$ is the wave with Cauchy data $u(x,0)=f_0(x),\;\partial_t u(x,0)=g(x)$, then the convolution solution \eqref{E:conv-soln1} shows that $u_1=f_1$ and $u_\alpha=f_\alpha$.  

We observe that because $g$ has compact support, \eqref{E:conv-soln1} also shows that $u_t\in\mathscr  D(\mathbb R^n)$ for all $t\in\mathbb R$.
\end{proof}




Theorem \ref{T:snapshot-compact-support} has the following application. If  $u(x,t)$ is a wave such that $u_0$ and $u_r$ belong to $\mathscr  D(\mathbb R^n)$ for a rational number $r \neq 0$, then we can only guarantee that $u_t$ belongs to $\mathscr  D(\mathbb R^n)$ for $t$ a multiple of $r$.  Indeed, the theorem above says that if $u_t\in\mathscr D(\mathbb R^n)$ for some irrational number $t$, then $u_t\in\mathscr D(\mathbb R^n)$ for all $t$;  this implies that if $u_t\in\mathscr E(\mathbb R^n)\setminus \mathscr  D(\mathbb R^n)$ for some irrational $t$, then $u_t\in \mathscr E(\mathbb R^n)\setminus \mathscr  D(\mathbb R^n)$ for all irrational $t$.

A simple illustration is given by the function $u(x,t)=\sin(\pi t)\,e^{i \pi\zeta\cdot x}$ where $\zeta\in\mathbb C^n$ satisfies $\zeta\cdot\zeta=1$.  The function $u$ clearly satisfies the wave equation, and it is also clear that $u_m\equiv 0$ for all $m\in\mathbb Z$; however  $u_t\notin\mathscr  D(\mathbb R^n)$ for all $t\notin\mathbb Z$.

\section{Extensions to Noncompact Symmetric Spaces}

Our results extend to shifted wave equations on noncompact symmetric spaces, with certain modifications to some proofs in \S3.  Let $X=G/K$ be a noncompact Riemannian symmetric space, where $G$ is a real noncompact connected semisimple Lie group with finite center, and $K$ is a maximal compact subgroup of $G$.  We let $o$ denote the identity coset $\{K\}$ in $X$.  We will adopt the symmetric space notation in Helgason's books \cite{DS}, \cite{GGA}, and \cite{GASS}, and we will summarize the background we need below.  

The Lie algebra $\mathfrak{g}$ of $G$ has Cartan decomposition $\mathfrak g=\mathfrak k+\mathfrak p$.  The standard  $G$-invariant Riemannian structure on $X$ is the one obtained by left-translating the Killing form $\langle\phantom{x},\phantom{y}\rangle$ on  $\mathfrak p$, which can be identified with the tangent space $T_oX$.

Let $\mathfrak a$ be a maximal abelian subspace of $\mathfrak p$, and let $\Sigma$ be the resulting set of restricted roots of $\mathfrak g$ with respect to $\mathfrak a$.  For each $\alpha\in\Sigma$, let $\mathfrak g_\alpha$ be the associated restricted root space, and let $m_\alpha\;( =\dim \mathfrak g_\alpha)$ denote its multiplicity.  Fix a Weyl chamber $\mathfrak a^+$ in $\mathfrak a$, let $\Sigma^+$ denote the corresponding set of positive restricted roots.  Then $\mathfrak n=\sum_{\alpha\in\Sigma^+} \mathfrak g_\alpha$ is a nilpotent Lie subalgebra of $\mathfrak g$.  If $A$ and $N$ are the analytic subgroups of $G$ with Lie algebras $\mathfrak{a}$ and $\mathfrak n$, respectively, then we have the Iwasawa decomposition $G=NAK$.

For each $g\in G$, we accordingly write $g=n(g)\,\exp A(g)\,k(g)$, where $n(g)\in N,\;A(g)\in\mathfrak a$, and $k(g)\in K$.  A \emph{horocycle} in $X$ is a translate of an $N$-orbit in $X$.  The left action of the group $G$ is transitive on the set of all horocycles, and the isotropy subgroup at the horocycle $\xi_0=N\cdot o$ is $MN$, where $M$ is the centralizer of $A$ in $K$. Thus the set $\Xi$ of all horocycles can be identified with the homogeneous space $G/MN$, which is diffeomorphic to the product manifold $K/M\times A$, via the map $(kM,a)\mapsto ka\cdot\xi_0$.  For the horocycle $\xi=ka\cdot\xi_0$, the element $kM\in K/M$ is called its \emph{normal}, and the element $\log a\in \mathfrak a$ is called the \emph{directed distance} from the identity coset $o$ in $X$ to the horocycle $\xi$.  This directed distance can be identified with $\mathbb R$ in case $X$ has rank one; i.e., $\dim A=1$.

If we fix $kM\in K/M$, then each $x\in X$ lies in a unique horocycle with normal $kM$; we denote the directed distance from $o$ to this horocycle by $A(x,kM)$.  For each $g\in G$, we then have $A(g\cdot o,kM)=A(k^{-1}g)$.

Let $M'$ denote the normalizer of $A$ in $K$.  Then the Weyl group $W$ corresponding to the root system $\Sigma$ is isomorphic (via the adjoint representation of $M'$ on $\mathfrak a$) to the quotient group $M'/M$.  The dual space $\mathfrak a^*$ can be identified with $\mathfrak a$ via the Killing form, and in this way $W$ acts on $\mathfrak a^*$ and the complexified dual $\mathfrak a^*_\mathbb C$.  We extend the Killing form $\langle\phantom{x},\phantom{y}\rangle$ to $\mathfrak a^*$ by duality, and we denote its extension to $\mathfrak a^*_\mathbb C$ also by $\langle\phantom{x},\phantom{y}\rangle$.

The \emph{Fourier transform} of any $f\in\mathscr D(X)$ is the function $\widetilde f$ on $\mathfrak a^*_{\mathbb C}\times K/M$ given by
\begin{equation}\label{E:fouriertransform1}
\widetilde f(\lambda,kM)=\int_X f(x)\,e^{(-i\lambda+\rho)A(x,kM)}\,dx.
\end{equation}
In the above, $dx$ is the Riemannian measure on $X$ and we have put $\rho=(1/2)\,\sum_{\alpha\in\Sigma^+} m_\alpha\,\alpha$.  If $f$ is left $K$-invariant, then $\widetilde f$ is constant in the second argument, and \eqref{E:fouriertransform1} reduces to the \emph{spherical Fourier transform}
\begin{equation}\label{E:sphericaltransform1}
\widetilde f(\lambda)=\int_X f(x)\,\varphi_{-\lambda} (x)\,dx,
\end{equation}
where $\varphi_\lambda$ is the spherical function corresponding to the spectral parameter $\lambda\in\mathfrak a^*_\mathbb C$, given by
$$
\varphi_\lambda(x)=\int_K e^{(i\lambda+\rho)A(x,kM)}\,dk,\qquad\qquad x\in X.
$$
Let $\mathscr E'(X)$ denote the space of compactly supported distributions on $X$, and let $\mathscr E'_K(X)$ denote the subspace of $\mathscr E'(X)$ consisting of left $K$-invariant distributions.  The spherical Fourier transform of any $S\in\mathscr E'_K(X)$ is defined by
\begin{equation}\label{E:sphericaltransform2}
\widetilde S(\lambda)=\int_X \varphi_{-\lambda}(x)\,dS(x),\qquad\qquad \lambda\in\mathfrak a^*_\mathbb C.
\end{equation}
There are Paley-Wiener theorems pertaining to the Fourier transforms \eqref{E:fouriertransform1} -- \eqref{E:sphericaltransform2} (see \cite{GGA}, Ch.~IV and \cite{GASS}, Ch.~III).  The one that is relevant to us is the following.
\begin{theorem}\label{T:K-inv-paleywiener}
For any $R>0$, the spherical Fourier transform $S\mapsto\widetilde S$ is a linear bijection from the subspace of $\mathscr E'_K(X)$ consisting of distributions supported in the closed ball $\overline B_R(o)$ and the vector space of $W$-invariant entire functions $F$ on $\mathfrak a^*_{\mathbb C}$ for which there is a constant $C$ and an integer $N$ such that
\begin{equation}\label{E:exponentialtype1}
|F(\lambda)|\leq C (1+|\lambda|)^N\,e^{R|\mathrm{Im}\lambda|},\qquad\qquad \lambda\in\mathfrak a^*_{\mathbb C}.
\end{equation}
\end{theorem}

We will also need to recall a few details about convolutions on $G/K$.  For this we just need to assume that $X=G/K$ is any homogeneous space, with $G$ unimodular and $K$ compact.  (For details on the convolution calculus on $G/K$, see \cite{GGA}, Chapter II.)  Suppose that $f\in \mathscr E(X)$ and $S\in \mathscr E'_K(X)$.  The convolution $f*S$ is defined by ``pulling back" the convolution to the group $G$.  To be precise, let $s$ denote the  pullback of $S$ to the group $G$.  This means that $s$ is the distribution on $G$ defined by
$$
s(\psi)=S(\psi^\natural),\qquad\qquad \psi\in \mathscr E(G).
$$
In the above, $\psi^\natural$ is the ``push-down''  of the function $\psi$ to $X$ given by right-averaging $\psi$ over $K$:
$$
\psi^\natural (u\cdot o)=\int_K \psi(uk)\,dk, \qquad u\in G.
$$
  Then $\psi^\natural\in\mathscr E(X)$, and the map $\psi\mapsto\psi^\natural$ is a continuous linear map from $\mathscr E(G)$ onto $\mathscr E(X)$ . By definition, the convolution $f*S$ is the function on $X$  given by
\begin{equation}\label{E:homogeneous-space-convolution}
(f*S) (u\cdot o)=\int_{G}  f(uv^{-1}\cdot o)\,ds(v),\qquad u\in G.
\end{equation}
It is easily ascertained that $f*S\in\mathscr E(X)$, and it can be shown that the right convolution operator $C_S\colon f\mapsto f*S$ is continuous from $\mathscr E(X)$ to $\mathscr E(X)$.  (The proof of this is a simple consequence of the closed graph theorem -- see, e.g., \cite{RudinFA}, Theorem 6.33.)

If $T\in\mathscr D'(X)$ and $S\in\mathscr E_K'(X)$, then their convolution $T*S$ is the distribution on $X$ defined as follows.  Let $\widecheck S$ be the distribution on $X$ defined by
$$
\widecheck S(\varphi)=\int_G \varphi(g^{-1}\cdot o)\,ds(g), \qquad\qquad \varphi\in\mathscr E(X).
$$
Then $\widecheck S\in\mathscr E'_K(X)$ and the map $S\mapsto\widecheck S$ is a linear homeomorphism of $\mathscr E_K'(X)$ onto itself.  We define $T*S$ by
\begin{equation}\label{E:homogeneous-space-convolution-2}
(T*S)(f)=T(f*\widecheck S),\qquad\qquad f\in\mathscr D(X),
\end{equation}
so that
$$
(T*S)(f)=\int_{G/K}\int_{G/K} f(ghK)\,dT(gK)\,dS(hK).
$$
It is not hard to prove that $S*T=T*S$ whenever $S,\,T\in\mathscr E'_K(X)$.  (See \cite{GGA}, Ch.~II, \S5.)

On our noncompact Riemannian symmetric space $X=G/K$, right convolution with $S$ corresponds to multiplication with $\widetilde S$ in the Fourier domain:
\begin{equation}\label{E:fourierconv}
(f*S)^\sim(\lambda,kM)=\widetilde f(\lambda, kM)\,\widetilde S(\lambda),
\end{equation}
for all $\lambda\in\mathfrak a^*_\mathbb C$ and $kM\in K/M$.  (See \cite{GASS}, Ch.~III, Sec.\,1.)  

Let $l=\dim \mathfrak a$.  Then the dual space $\mathfrak a^*_\mathbb C$ can be identified with $\mathbb C^l$ by fixing a basis of $\mathfrak a^*$. Using this identification, we can then use Definition \ref{D:slow-decrease} to define what it means for a function on $\mathfrak a^*_\mathbb C$ to be slowly decreasing.  It is easy to see that this definition is independent of the choice of basis of $\mathfrak a^*$.  Ehrenpreis' surjectivity result has the following counterpart for convolutions on noncompact symmetric spaces.
\begin{theorem}\label{T:symconvsurjectivity}
Let $S\in\mathscr E_K'(X)$. Then the convolution operator $C_S\colon \mathscr E(X)\to\mathscr E(X)$ is surjective if and only if $\widetilde S$ is slowly decreasing in $\mathfrak a^*_\mathbb C$.
\end{theorem}
The sufficiency of the slow decrease condition is proved in \cite{CGK2017}, Theorem 5.1, and the necessity in \cite{GKW2021}, Theorem 3.1.  The operator $C_S$ is in general not injective; for example, if $\lambda_0\in\mathfrak a^*_\mathbb C$ is a zero of $\widetilde S(\lambda)$, then $C_S(\varphi_{\lambda_0})=\widetilde S(\lambda_0)\,\varphi_{\lambda_0}=0$.

 After the above preliminaries, we can now consider the snapshot problem for the symmetric space $X$.

Let $\Delta$ denote the Laplace-Beltrami operator on $X$.  Any solution $u(x,t)\in C^\infty(X\times\mathbb R)$ of the \emph{shifted wave equation}
\begin{equation}\label{E:modifiedwaveeqn}
\Delta u(x,t)=\left(\frac{\partial^2}{\partial t^2} - \langle\rho,\rho\rangle\right) u(x,t),
\end{equation}
will be called a \emph{wave} in $X$.
If we are given the Cauchy data
\begin{equation}\label{E:symmcauchydata}
u(x,0)=f(x),\qquad (\partial_t u)(x,0)=g(x),
\end{equation}
then $u$ is unique (since \eqref{E:modifiedwaveeqn} is hyperbolic),  and is given by
\begin{equation}\label{E:convsolnsym}
u(x,t) = f*S_t'(x) + g* S_t(x),
\end{equation}
where $S_t$ and $S_t'$ are $K$-invariant compactly supported distributions on $X$ whose spherical Fourier transforms are given by
\begin{equation}\label{E:symfundsoln}
\widetilde S_t(\lambda) = \frac{\sin t \sqrt{\langle\lambda,\lambda\rangle}}{\sqrt{\langle\lambda,\lambda\rangle}},\qquad \widetilde S_t'(\lambda) = \cos t\sqrt{\langle\lambda,\lambda\rangle},\qquad \lambda\in\mathfrak a^*_{\mathbb C}.
\end{equation}
We note that $\widetilde S_t$ and $\widetilde S_t'$ are $W$-invariant entire functions on $\mathfrak a^*_\mathbb C$ satisfying estimates of the form \eqref{E:exponentialtype1}, with $R=|t|$. By Theorem \ref{T:K-inv-paleywiener}, $S_t$ and $S_t'$ are thus $K$-invariant distributions on $X$ supported in the closed ball $\overline B_{|t|}(o)$ in $X$.  We call the family of pairs $(S_t,\,S_t')$ the \emph{fundamental solution} of \eqref{E:modifiedwaveeqn}.

To see that \eqref{E:convsolnsym} is the solution of the Cauchy problem \eqref{E:modifiedwaveeqn} - \eqref{E:symmcauchydata}, we first recall that if $S\in\mathscr E_K'(X)$, then
$$
\left(\Delta S\right)^\sim(\lambda)=\left(-\langle\lambda,\lambda\rangle-\langle\rho,\rho\rangle\right)\,\widetilde S, \qquad \lambda\in\mathfrak a^*_\mathbb C.
$$
(See \cite{GGA}, Ch.~IV.)   Given \eqref{E:symfundsoln}, this implies that $S_t$ and $S_t'$ satisfy the shifted wave equations
$$
\Delta S_t=\left(\frac{\partial^2}{\partial t^2}-\langle\rho,\rho\rangle\right)\,S_t.
$$   
Since $\Delta(u*S_t)=u*\Delta S_t$ and $\Delta(u*S_t')=u*\Delta S_t'$, it follows that the $u(x,t)$ in \eqref{E:convsolnsym} satisfies \eqref{E:modifiedwaveeqn}.  That this $u$ satisfies the initial conditions \eqref{E:symmcauchydata} follows from the fact that $(S_0')^\sim=1,\;(\partial_t)(S_t')^\sim|_{t=0}=0$, and $(\partial_t)(S_t)^\sim|_{t=0}=1$.

Suppose that $u(x,t)$ is a wave on $X$.  For any $t\in \mathbb R$, we let $u_t\in \mathscr E(X)$ be the function $u_t(x)=u(x,t)$, and call $u_t$ the \emph{snapshot of $u$ at time $t$}.

Note that the spherical Fourier transforms $\widetilde S_t$ and $\widetilde S_t'$ have the exact same form as their Euclidean counterparts \eqref{E:wave-fund-soln}.  In view of the fact that these are slowly decreasing functions, and in view of the multiplicativity result \eqref{E:fourierconv}, many of the results in \S3 can be restated for the snapshot problem for $X=G/K$ with just a few modifications to the proofs.  These modifications are necessary mainly because the Fourier transforms on $X$ are given in the ``polar coordinates'' $\mathfrak a^*_{\mathbb C}\times K/M$.

For every $\mu\in\mathbb C$, let $\mathscr E_\mu(X)$ denote the eigenspace of $\Delta$ consisting of all $f\in\mathscr E(X)$ such that $\Delta f=-(\mu^2+\langle \rho,\rho\rangle)\, f$.  (Unless $X$ has rank one, these eigenspaces are larger than the joint eigenspaces of all $G$-invariant differential operators on $X$.)  
\begin{lemma}\label{T:eigenspace2}
If $f\in\mathscr E_\mu(X)$, then
$$
f*S_t=\frac{\sin \mu t}{\mu}\,f.
$$
\end{lemma}
Just as in the proof of Lemma \ref{T:eigenspace1}, the proof follows from observing that both sides above are solutions of the shifted wave equation \eqref{E:modifiedwaveeqn} with the same initial data $u(x,0)=0,\;(\partial_t u)(x,0)=f(x)$.

\begin{theorem}\label{T:kernel2}
For any $t\neq 0$, the kernel of the convolution operator $C_{S_t}$ is the closure in $\mathscr E(X)$ of the direct sum
$$
\bigoplus_{j=1}^\infty \mathscr E_{\pi j/t}(X).
$$
\end{theorem}
\begin{proof}
By Lemma \ref{T:eigenspace2} for each $j=1,2,\ldots$, the eigenspace $\mathscr E_{\pi j/t}(X)$ belongs to the kernel of $C_{S_t}$.  Hence the closure of the above direct sum is a subspace of $\ker (C_{S_t})$.  Conversely, suppose that $f\in\ker(C_{S_t})$ but does not belong to the closure of the direct sum above.  Then there exists a $T\in\mathscr E'(X)$ such that $T=0$ on the direct sum above, but $T(f)\neq 0$.  

For any positive integer $j$, let $V_j$ be the algebraic variety in $\mathfrak a^*_\mathbb C$ consisting of all $\lambda$ such that $\langle\lambda,\lambda\rangle=\pi^2 j^2/t^2$.  If $\lambda\in V_j$ and $b\in K/M$, then the Fourier exponential 
$$
e_{-\lambda,b}(x)=e^{(-i\lambda+\rho)A(x,b)}
$$
belongs to the eigenspace $\mathscr E_{\pi j/t}$. (See \cite{GASS}, Ch.~II, Prop.~3.14.)  Applying $T$ to this function, we conclude that $\widetilde T(\lambda,b)=0$.  

Thus for each $b\in K/M$, the function $\lambda\mapsto \widetilde T(\lambda,b)$ is a holomorphic function on $\mathfrak a^*_\mathbb C$ that vanishes on the $V_j$.  The union of the $V_j$ is precisely the zero set of $\widetilde S_t$, and the complex differential $\partial \widetilde S_t$ is never zero at any point in this union. Hence for any $b\in K/M$, the function $\lambda\mapsto\widetilde T(\lambda, b)/\widetilde S_t(\lambda)$ is holomorphic in $\mathfrak a^*_{\mathbb C}$.  The proof of \cite{CGK2017}, Theorem 5.1 then shows that there exists a distribution $\Psi\in\mathscr E'(X)$ such that $T=\Psi*S_t$.

This implies that 
$$
T(f)=\Psi*S_t(f)=\psi(f*\widecheck S_t)=\Psi (f*S_t)=0,
$$
a contradiction.  In the above we have used the fact that $\widecheck S_t=S_t$, which follows from the fact that $(\widecheck S_t)^\sim(\lambda)=\widetilde S_t(-\lambda)=\widetilde S_t(\lambda)$, for all $\lambda\in\mathfrak a^*_\mathbb C$.
\end{proof}

This brings us to the symmetric space analogue of Proposition \ref{T:two-snapshot-existence}.  
\begin{proposition}\label{T:two-snapshot-existence2}
Given $f,h\in\mathscr E(X)$, there exist infinitely many waves $u(x,t)$ such that $u_0=f$ and $u_1=h$.
\end{proposition}

Observing that $\widetilde S_t(\lambda)$ is a slowly decreasing function on $\mathfrak a^*_{\mathbb C}$, 
the proof is the same as that of Proposition \ref{T:two-snapshot-existence}, in view of Theorem \ref{T:symconvsurjectivity} and Theorem \ref{T:kernel2}.  The proposition above easily generalizes, with the same proof, to the result that there exist infinitely many waves with prescribed snapshots $u_a=f$ and $u_b=h$, at any two given times $a$ and $b$.

Although there are infinitely many waves $u(x,t)$ with given snapshots $u_0$ and $u_1$, the symmetric space analogues of Theorems \ref{T:integersnapshot} and \ref{T:integersnapshot2} also hold: the integer time snapshots $u_m$ are uniquely determined, for any $m\in\mathbb Z$.
\begin{theorem}\label{T:symintegersnapshot}
Suppose that $u(x,t)\in C^\infty(X\times \mathbb R)$ is a solution of the shifted wave equation \eqref{E:modifiedwaveeqn}. Then for any $m\in\mathbb Z$, the snapshot $u_m$ of $u$ is determined completely by $u_0$ and $u_1$. More precisely, we have
\begin{equation}\label{E:intsnapshotsym}
u_m=u_1*\Psi_m-u_0*\Psi_{m-1},
\end{equation}
where the distribution $\Psi_m\in\mathscr E_K'(X)$ is given by
\begin{equation}\label{E:psi-m2}
\widetilde\Psi_m(\lambda)=\frac{\sin(m\sqrt{\langle\lambda,\lambda\rangle})}{\sin \sqrt{\langle\lambda,\lambda\rangle}}=U_{m-1}(\cos \sqrt{\langle\lambda,\lambda\rangle})\qquad (\lambda\in\mathfrak{a}^*_\mathbb C).
\end{equation}
Here $U_m$ represents the Chebyshev polynomial of the second kind, with the supplementary definition $U_{-m-1}=-U_{m-1}$ for $m\in\mathbb Z^+$.

More generally, suppose that $a<b$ and we are given the snapshots $u_a$ and $u_b$, where $a<b$.  If we put $s=b-a$, then
\begin{equation}\label{E:general-integer-time2}
u_{a+m(b-a)}=u_b*\Psi_{m,s}-u_a*\Psi_{(m-1),s}, \qquad\qquad m\in\mathbb Z,
\end{equation}
where $\Psi_{m,s}\in\mathscr E_K'(X)$ is given by
\begin{equation}\label{E:psi-ms2}
\widetilde\Psi_{m,s}(\lambda)=\frac{\sin(ms\sqrt{\langle\lambda,\lambda\rangle})}{\sin (s\sqrt{\langle\lambda,\lambda\rangle})}=U_{m-1}(\cos s\sqrt{\langle\lambda,\lambda\rangle}),\qquad\qquad \lambda\in\mathfrak a^*_\mathbb C.
\end{equation}
\end{theorem}
\begin{proof}
The proof is the same as that of Theorem \ref{T:integersnapshot}, if we take into account the multiplicative property \eqref{E:fourierconv} of right convolution by elements of $\mathscr E'_K(X)$.
\end{proof}

Next we consider the existence and uniqueness question for a wave on $X$ with snapshots at three given times, which for convenience we specify as $t=0,1$, and $\alpha$.  (The reason we may do so is that for any positive constant $c$, we can rescale the Riemannian structure on $X$ by a factor of $c^2$ and the time axis by a factor of $c$.  This replaces $\Delta$ by $(1/c^2)\,\Delta$, $\partial ^2/\partial t^2$ by $(1/c^2)\,\partial^2/\partial t^2$, and $\langle \rho,\rho\rangle$ by $(1/c^2)\,\langle \rho,\rho\rangle$.  Under these changes, waves remain the same but the time scales are different.)

Uniqueness is easily established: due to Proposition \ref{T:integersnapshot2} being the exact analogue of Theorems \ref{T:integersnapshot} and \ref{T:integersnapshot2}, we can apply the same reasoning which led to Proposition \ref{T:snapshot-uniqueness} to obtain the following result.

\begin{proposition}\label{T:symsnapshot-uniqueness}  
If $\alpha$ is a irrational number, then a wave $u(x,t)$ is uniquely determined by its snapshots $u_0$, $u_1$ and $u_\alpha$.  
If $\alpha$ is a nonzero rational number with $\alpha=p/q$ in lowest terms, then $u$ can be only be determined up to its snapshots $u_t$ where $t$ is an integer multiple of $1/q$.
\end{proposition}

As for existence, suppose we are given functions $f_0,\,f_1$, and $f_\alpha$ in $\mathscr E(X)$.  In analogy with \eqref{E:existenceofg},  the convolution solution \eqref{E:convsolnsym} shows that there is a wave $u(x,t)$ such that $u_0=f_0,\,u_1=f_1$, and $u_\alpha=f_\alpha$ if and only if there exists a $g\in\mathscr E(X)$ such that 
\begin{equation}\label{E:symjointpair}
g*S_1=f_1-f_0*S_1'\quad\text{and}\quad g*S_\alpha=f_\alpha-f_0*S_\alpha'.
\end{equation}
This leads to the compatibility condition
\begin{equation}\label{E:symcompatibility}
(f_1-f_0*S_1')*S_\alpha=(f_\alpha-f_0*S_\alpha')*S_1,
\end{equation}
which is necessary for the existence of such a $g$.

Again, as with condition \eqref{E:compatibility-three}, the above condition has the equivalent formulation 
$$
f_0*S_{\alpha-1}+f_1*S_{-\alpha}+f_\alpha*S_1=0.
$$

Now the convolution operators $C_{S_1}$ and $C_{S_\alpha}$ are surjective and commute on $\mathscr E(X)$.  According to Proposition \ref{T:linalg1}, the compatibility condition \eqref{E:symcompatibility} guarantees the existence of a $g\in\mathscr E(X)$ satisfying the convolution equations \eqref{E:symjointpair} if and only if $C_{S_\alpha}(\ker (C_{S_1}))=\ker (C_{S_1})$.  This is not possible if $\alpha$ is a Liouville number, as the following analogue to Lemma \ref{T:liouville1} shows.
\begin{lemma}\label{T:liouville4}
Suppose that $\alpha$ is a Liouville number.  Then the operator $C_{S_\alpha}\colon \ker(C_{S_1})\to\ker(C_{S_1})$ is not surjective, and hence its range is of first category in $\ker(C_{S_1})$.
\end{lemma}
\begin{proof}
If a continuous linear map of Fr\'echet spaces is not surjective, then its range is of first category, so the second assertion in the lemma follows easily from the first. 

Now assume to the contrary that $C_{S_\alpha}\colon \ker(C_{S_1})\to\ker(C_{S_1})$ is surjective.  An argument similar to that of Proposition \ref{T:csalphainjectivity} shows that it is injective.  Hence it is a homeomorphism, so the inverse map $C_{S_\alpha}^{-1}\colon \ker(C_{S_1})\to\ker(C_{S_1})$ is continuous.  For each $j\in\mathbb N$,  Theorem \ref{T:kernel2} shows that the eigenspace $\mathscr E_{\pi j}$ belongs to $\ker(C_{S_1})$, and by Lemma \ref{T:eigenspace2} the restriction of the linear operator $C_{S_\alpha}^{-1}$ to $\mathscr E_{\pi j}$ is just multiplication by $\pi j/\sin (\pi j\alpha)$. 

Since $\alpha$ is a Liouville number, there exists a sequence $\{q_k\}$ of positive integers $\geq 2$ such that 
\begin{equation}\label{E:liouville-sine-estimate}
|\sin (\pi q_k\alpha)|\leq \pi q_k^{-k}
\end{equation}
 for all $k$.  We will now produce a sequence $\{f_k\}$ in $\ker(C_{S_1})$ that converges to $0$ (in the topology of $\mathscr E(X)$) such that $\{C_{S_\alpha}^{-1}(f_k)\}$ does not converge to $0$, which will give us a contradiction.

For this, fix an orthonormal basis $H_1,\ldots H_l$ of $\mathfrak a$ and a basis $Y_1,\ldots,Y_m$ of $\mathfrak n$.  Since $(n,a)\mapsto na\cdot o$ is a diffeomorphism from $N\times A$ onto $X$, the map
\begin{equation}\label{E:globalcoordinates}
\exp(s_1X_1+\cdots+s_mX_m)\,\exp(t_1H_1+\cdots+t_lH_l)\cdot o\mapsto (s_1,\ldots,s_m,t_1,\ldots,s_l)
\end{equation}
is a global coordinate system on $X$.  This coordinate system (in fact any global coordinate system on $X$) is useful because it allows us to identify the topological vector spaces $\mathscr E(X)$ and $\mathscr E(\mathbb R^n)$.

Let $\lambda_1,\ldots,\lambda_l$ be a basis of $\mathfrak a^*$ dual to $H_1,\ldots,H_l$.  Then $\rho=r_1\lambda_1+\cdots+r_l\lambda_l$, for some real constants $r_1,\ldots,r_l$ and if $\lambda\in\mathfrak a^*_\mathbb C$, then we have $\lambda=\zeta_1\lambda_1+\cdots+\zeta_l\lambda_l$, for complex constants $\zeta_1,\ldots,\zeta_l$. If $x=na\cdot o\in X$, with $a=\exp (t_1H_1+\cdots+t_lH_l)$, the Fourier exponential $e_{\lambda,eM}(x)=e^{(i\lambda+\rho)A(x,eM)}=e^{(i\lambda+\rho)\log a}$ equals
$$
e^{(i \zeta_1+r_1)t_1+\cdots+(i \zeta_l+r_l)t_l}.
$$

For each $k\in\mathbb N$, consider the function
\begin{equation}\label{E:fourier-exponential}
g_k(x)=e_{\pi q_k\lambda_1,eM}(x)=e^{(i\pi q_k\lambda_1+\rho)A(x,eM)}=e^{i\pi q_k t_1}\cdot e^{\sum r_i t_i}.
\end{equation}
Then $g_k$ belongs to the eigenspace $\mathscr E_{\pi q_k}$, and hence lies in $\ker(C_{S_1})$.  Next let $f_k=g_k/q_k^k$.  The sequence $\{f_k\}$ converges to $0$ in $\mathscr E(X)$; in fact, using the global coordinates \eqref{E:globalcoordinates}, \eqref{E:fourier-exponential} shows that for any multiindices $\alpha=(\alpha_1,\ldots,\alpha_l)$ and $\beta=(\beta_1,\ldots,\beta_m)$ and any compact subset $S$ of $X$, there is a constant $C$ such that
$$
\left|\frac{\partial^{|\alpha|}}{\partial t_1^{\alpha_1}\cdots\partial t_l^{\alpha_l}} \frac{\partial^{|\beta|}}{\partial s_1^{\beta_1}\cdots\partial s_m^{\beta_m}} f_k(x)\right|\leq C \frac{q_k^{\alpha_1}}{q_k^k} \to 0\quad\text{as }k\to\infty.
$$
On the other hand, by Lemma \ref{T:eigenspace2},
$$
C_{S_\alpha}^{-1}(f_k)=\frac{\pi (\sin(\pi q_k\alpha))^{-1}}{q_k^{k-1}}\,g_k.
$$
By the estimate  \eqref{E:liouville-sine-estimate}, we obtain $|C_{S_\alpha}^{-1}(f_k)(o)|\geq q_k$ for all $k$, and we get our contradiction.
\end{proof}
Assuming that $\alpha$ is a Liouville number, the fact that $C_{S_\alpha}(\ker C_{S_1})$ is a set of first category in the Fr\'echet space $\ker(C_{S_1})$ implies that ``most" triples $(f_0,f_1,f_\alpha)$ which satisfy the compatibility condition \eqref{E:symcompatibility} are not wave snapshots at times $t=0,1,\alpha$.  To be explicit, let us fix the functions $f_0$ and $f_1$, and also fix a wave $u(x,t)$ for which $u_0=f_0,\,u_1=f_1$.  (There are of course many such $u$.)  Then by Remark \ref{R:linalgremark} and the proof of Theorem \ref{T:liouville1}, a function $f_\alpha\in\mathscr E(X)$  satisfies the compatibility condition \eqref{E:symcompatibility} with $f_0$ and $f_1$ if and only if $f_\alpha-u_\alpha\in\ker(C_{S_1})$.  On the other hand, for any $f_\alpha\in\mathscr E(X)$, $f_\alpha$ is a wave snapshot together with $f_0$ and $f_1$ if and only if $f_\alpha-u_\alpha\in C_{S_\alpha}(\ker(C_{S_1})$.

Next let us suppose that $\alpha$ is not a Liouville number and is irrational.  Recall that almost all real numbers are of this type.  Since the spherical Fourier transforms $\widetilde S_\alpha$ and $\widetilde S'_\alpha$ in \eqref{E:symfundsoln} have the exact same expression as their Euclidean counterparts \eqref{E:wave-fund-soln}, the estimate \eqref{E:fund-estimate1} holds, and we can apply Theorem \ref{T:Ehrenpreis-estimate} to conclude that there exist holomorphic functions $\phi$ and $\psi$ on $\mathfrak a^*_\mathbb C$ which are polynomially increasing on $\mathfrak a^*$ and of exponential type, such that
$$
\phi(\lambda)\,\widetilde S_1(\lambda)+\psi(\lambda)\,\widetilde S_\alpha(\lambda)=1.
$$
Averaging the left hand side above over the Weyl group $W$, and noting that $\widetilde S_1$ and $\widetilde S_\alpha$ are $W$-invariant, we can replace $\phi$ and $\psi$ by their $W$-averages.  We can therefore assume that $\psi$ and $\psi$ are $W$-invariant.  By Theorem \ref{T:K-inv-paleywiener}, there exist distributions $\Phi$ and $\Psi$ in $\mathscr E_K'(X)$ such that $\widetilde\Phi=\phi$ and $\widetilde\Psi=\psi$.  It follows that
$$
\Phi*S_1+\Psi*S_\alpha=\delta_o.
$$
Just as with Lemma \ref{T:kernelsum}, it follows that if functions $f_0,f_1$, and $f_\alpha$ in $\mathscr E(X)$ satisfy the compatibility condition \eqref{E:symcompatibility}, then $(f_1-f_0*S_1',f_\alpha-f_0*S_\alpha')$ is a joint pair for $(S_\alpha,S_1)$.  Hence there is a $g\in\mathscr E(X)$ such that $g*S_1=f_1-f_0*S_1'$ and $g*S_\alpha=f_\alpha-f_0*S_\alpha'$.  We obtain the following analogue of Theorem \ref{T:irrational-not-Liouville}.

\begin{theorem}\label{T:sym-irrational-not-Liouville}
	Let $\alpha$ be irrational and not a Liouville number. If $f_0,\,f_1$ and $f_\alpha$ are functions in $\mathscr E(X)$ satisfying the compatibility condition \eqref{E:symcompatibility}:
	$$ (f_1-f_0*S_1')*S_\alpha=(f_\alpha-f_0*S_\alpha')*S_1, $$
	then there exists a unique smooth wave $u(x,t)$ on $X$ whose snapshots at times $t=0,1,\alpha$ are precisely the given triple $(f_0,f_1,f_\alpha)$.
\end{theorem}

Finally, we consider the existence problem for waves with three snapshots at rational times.  Again, by rescaling the Riemannian structure on $X$ and translating and rescaling the time axis, we can assume that the times are $t=0,p,q$, where $p$ and $q$ are positive integers which are relatively prime.

The arguments in the proof of Theorem \ref{T:psi-compatibility} then carry through without modification in the case of the symmetric space $X$, and we have the following result.

\begin{theorem}\label{T:sym-psi-compatibility}  Let $p$ and $q$ be positive integers which are relatively prime, and suppose that $f_0,\,f_p,\,f_q\in \mathscr E(X)$ satisfy the condition
\begin{equation}\label{E:sym-integercompatibility}
(f_p-f_0*S_p')*\Psi_q=(f_q-f_0*S_q')*\Psi_p.
\end{equation}
Then there exists a wave $u$ such that $u_0=f_0,\,u_p=f_p$, and $u_q=f_q$.
\end{theorem}

\section{The Snapshot Problem for Spheres}

In this section we will consider the snapshot problem for a shifted wave equation on the unit $n$-sphere $S^n$ in $\mathbb R^{n+1}$ which was first considered by Lax and Phillips in 1978 \cite{LaxPhillips1978}.  Among other things, our purpose here is to show that for almost all $\alpha\neq 0$, there is a unique wave $u(x,t)$ matching given snapshots at just the two times $t=0$ and $t=\alpha$.

Before considering the problem in detail, it will be necessary to review some of the essential elements of harmonic analysis on $S^n$.   

For convenience we let $o=(0,0,\ldots,1)^T$, the last standard basis vector in $\mathbb R^{n+1}$, which we will call the \emph{north pole}.  The special orthogonal group $\text{SO}(n+1)$ acts transitively on $S^n$ by left multiplication, and the isotropy subgroup $K$ at $o$ is isomorphic to $\text{SO}(n)$; $S^n$ can thus be identified with the homogeneous space $G/K$, where $G=\text{SO}(n+1)$ and $K=\text{SO}(n)$.

We endow $S^n$ with the Riemannian metric inherited from $\mathbb R^{n+1}$, and let $\Delta_{S^n}$ denote the corresponding Laplace-Beltrami operator. 

For each nonnegative integer $l$, let $\mathcal H_l$ be the vector space of degree $l$ spherical harmonics on $S^n$.  $\mathcal H_l$ consists of the restrictions to $S^n$ of the degree $l$ homogeneous harmonic polynomials on $\mathbb R^{n+1}$. Each $\mathcal H_l$ is a joint eigenspace of $\Delta_{S^n}$:
\begin{equation}\label{E:harmoniceigenfcn}
\Delta_{S^n} \psi = -l(l+n-1)\,\psi,\qquad\qquad \psi\in \mathcal H_l.
\end{equation}

It is well known (see, e.g.~\cite{Muller1966}) that the dimension $d(l)$ of $\mathcal H_l$ is given by
$$
d(l)=\frac{(2l+n-1)\,\Gamma(l+n-1)}{\Gamma(l+1)\,\Gamma(n)},
$$
which is a polynomial in $l$ of degree $n-1$. From now on, we fix an orthonormal basis $\{Y_{lm}\}_{m=1}^{d(l)}$ (in $L^2(S^n))$ of each $\mathcal H_l$.  The $Y_{lm}$ can be taken to be real-valued and the first vector $Y_{l1}$ of the orthonormal basis can be taken to be invariant under rotations preserving the north pole $o$.   Then $Y_{l1}$ is proportional to the unique $K$-spherical function $\varphi_l$ in $\mathcal H_l$, and since $\|\varphi_l\|_2=(d(l))^{-1/2}$  (see \cite{GGA}, Ch.~V), we have
$$
Y_{l1}=\sqrt{d(l)}\,\varphi_l.
$$
The spherical function $\varphi_l$ is given by $\varphi_l(x)=P_k^{(n-1)/2}(x\cdot o)$, where $P_k^{(n-1)/2}$ is the Gegenbauer (or ultraspherical) polynomial of degree $l$ in $n$ dimensions. (The literature is certainly vast -- see, for example, \cite{SteinWeiss71}, Ch.~IV or \cite{Muller1966}.)

One of the defining characteristics of the spherical function $\varphi_l$ is the following mean value relation which holds for any function $F\in\mathcal H_l$:
\begin{equation}\label{E:sphericalmean}
\int_K F(uk\cdot y)\,dk=F(u\cdot o)\,\varphi_l(y),
\end{equation}
for all $u\in\text{SO}(n+1)$ and all $y\in S^n$.  (See \cite{GGA}, Ch.~IV.)  In the above, $dk$ denotes the normalized Haar measure on $K$.

Let $\langle\cdot\,,\,\cdot\rangle$ denote the inner product in $L^2(S^n)$.  Any function $f\in L^2(S^n)$, has a Fourier expansion
\begin{equation}\label{E:sphericalharmonicexpansion}
f=\sum_{l\geq 0}\sum_{m=1}^{d(l)} a_{lm}\,Y_{lm},
\end{equation}
with $a_{lm}=\langle f, Y_{lm}\rangle$.  It is well known that the function $f$ belongs to $C^\infty(S^n)$ if and only if its Fourier coefficients $a_{lm}$ are rapidly decreasing in $l$, that is to say
\begin{equation}\label{E:coeffrapiddecrease}
\sup_{l,m} |a_{lm}|\,l^N<\infty
\end{equation}
for all positive integers $N$.  (See, for example, \cite{Sugiura1971}, Corollary to Theorem 1.) 

It will be convenient for us to go over the topology of $\mathscr E(S^n)$.  As mentioned in \S1, for any manifold $M$, the locally convex topologies on the spaces $\mathscr D(M)$ and $\mathscr E(M)$ are discussed in detail in \cite{GGA}, Ch.~I, Sec.~2.  If $M$ is compact, the two spaces coincide and $\mathscr E(M)$ is a Fr\'echet space whose topology is defined by the seminorms 
$$
\|f\|_D=\|Df(x)\|_\infty,\qquad\qquad f\in C^\infty(M),
$$ 
where $D$ is any linear differential operator on $M$ with $C^\infty$ coefficients.  For the sphere $S^n$, it turns out that we only need powers of the Laplacian to obtain this topology.
\begin{theorem}\label{T:spheresmoothtopology}
The topology of $\mathscr E(S^n)$ coincides with the topology defined by the family of seminorms
$$
R_N (f)=\|(\Delta_{S^n})^N f\|_\infty, \quad f\in\mathscr E(S^n),\;N=0,1,2,\ldots.
$$
\end{theorem}
(See \cite{Sugiura1971}, Corollary to Theorem 4.)

From \eqref{E:harmoniceigenfcn} and \eqref{E:coeffrapiddecrease}, we see that the Fourier expansion \eqref{E:sphericalharmonicexpansion} of any $f\in\mathscr E(S^n)$ converges to $f$ in the topology of $\mathscr E(S^n)$.

Let $\mathscr E'(S^n)$ denote the space of distributions on $S^n$, and let $\mathscr E'_K(S^n)$ denote the space of $K$-invariant elements in $\mathscr E(S^n)$.  For any $S\in\mathscr E_K'(S^n)$, we define
\begin{equation}\label{E:schur-const}
    \widehat S(l)=S(\varphi_l)\qquad (l=0,1,2,\ldots)
\end{equation}
Then $S$ has a formal expansion in terms of the $K$-spherical functions $\varphi_l$ given by
\begin{equation}\label{E:spherical-expansion}
S=\sum_{l=0}^\infty d(l)\,\widehat S(l)\,\varphi_l.
\end{equation}

From Theorem \ref{T:spheresmoothtopology}, equation \eqref{E:harmoniceigenfcn}, and the fact that $\|\varphi_l\|_\infty=\varphi_l(o)=1$ for all $l$, we see that the coefficients $\widehat S(l)$ are polynomially increasing, in the sense that there is an integer $N>0$ and a constant $C$ such that
\begin{equation}\label{E:sphere-paleywiener}
    |\widehat S(l)|\leq C(1+l)^N \qquad\text{for all } l.
\end{equation}
Conversely, any polynomially increasing sequence $\{\beta_l\}$ gives rise to a unique distribution $S\in\mathscr E'_K(S^n)$ for which $\widehat S(l)=\beta_l$ for all $l$. (This can be regarded as the $K$-invariant Paley-Wiener Theorem on $S^n$.)  Indeed, for any $f\in\mathscr E(S^n)$ with Fourier expansion \eqref{E:sphericalharmonicexpansion}, define
$$
S(f)=\sum_{l\geq 0}\sum_{m=1}^{d(l)} \beta_l\,a_{lm}\,Y_{lm}(o).
$$
Then the series above converges absolutely (since $\beta_l\,a_{lm}$ is still rapidly decreasing in $l$), and the closed graph theorem shows that $S\in\mathscr E'(S^n)$.  Moreover, $S$ is clearly $K$-invariant, and it is also clear that $\widehat S(l)=\beta_l$.

We will now need some details about convolutions on $S^n$.  Suppose that $f\in\mathscr E(S^n)$, $T\in\mathscr E'(S^n)$, and $S\in\mathscr E'_K(S^n)$.  Then $f*S$ is defined as in \eqref{E:homogeneous-space-convolution}, and $T*S$ is defined as in \eqref{E:homogeneous-space-convolution-2}.

For any $S\in\mathscr E'_K(S^n)$, let $C_S$ denote the (right) convolution operator on $\mathscr E(S^n)$ given by $f\mapsto f*S$.  As stated in \S4, $C_S$ is a continuous linear operator on $\mathscr E(S^n)$. We will now need the following well-known formula.  
\begin{lemma}\label{T:interwining}  Let $S\in\mathscr E'_K(S^n)$. Then
\begin{equation}\label{E:convolution-fourier}  
\langle C_S(f), \psi\rangle=\widehat S(l)\,\langle f,\psi\rangle,
\end{equation}
for all $f\in\mathscr E(S^n)$ and all $\psi\in\mathcal H_l$.
\end{lemma}
\begin{proof} There are several ways to prove \eqref{E:convolution-fourier}.  The following is one that just depends on the group-theoretic definition \eqref{E:homogeneous-space-convolution} of convolution.

Let $du$ denote the normalized Haar measure on $\text{SO}(n+1)$, $dk$ the normalized Haar measure on $K=\mathrm{SO}(n)$, and let $\omega_n$ denote the area of $S^n$ in $\mathbb R^{n+1}$.  If $s\in\mathscr E'(\mathrm{SO}(n+1))$ is the ``pullback'' of $S$ to $\mathrm{SO}(n+1)$, then \eqref{E:homogeneous-space-convolution} and the Fubini theorem for distributions (see, for example, \cite{RudinFA}, Theorem 3.27) give us
\begin{align}
\langle f*S,\psi\rangle&=\omega_n\,\int_{\text{SO}(n+1)} \left(\int_{\text{SO}(n+1)} f(uv^{-1}\cdot o)\,ds(v)\right)\,\overline\psi(u\cdot o)\,du\nonumber\\
&=\omega_n\,\int_{\text{SO}(n+1)} \left(\int_{\text{SO}(n+1)} f(uv^{-1}\cdot o)\,\overline\psi(u\cdot o)\,du\right)\,ds(v)\nonumber\\
&=\int_{\text{SO}(n+1)} \left(\omega_n\,\int_{\text{SO}(n+1)} f(u\cdot o)\,\overline\psi(uv\cdot o)\,du\right)\,ds(v)\label{E:conv-fouriercoeff}
\end{align}
Now the inner integral above equals
\begin{align*}
\omega_n\,\int_{\text{SO}(n+1)} &\left(\int_K  f(uk\cdot o)\,dk\right)\,\overline\psi(uv\cdot o)\,du\\
&=\omega_n\,\int_K\left(\int_{\text{SO}(n+1)} f(u\cdot o)\,\overline\psi(ukv\cdot o)\,du\right)\,dk\\
&=\omega_n\,\int_{\text{SO}(n+1)}f(u\cdot o)\,\left(\int_K \overline\psi(ukv\cdot o)\,dk\right)\,du\\
&=\varphi_l(v\cdot o)\,\omega_n \int_{\text{SO}(n+1)} f(u\cdot o)\,\overline\psi(u\cdot o)\,du\\
&=\varphi_l(v\cdot o)\,\langle f,\psi\rangle,
\end{align*}
where we have used \eqref{E:sphericalmean} and the fact that the spherical function $\varphi_l$ is real-valued.  Plugging this back into \eqref{E:conv-fouriercoeff}, the formula \eqref{E:convolution-fourier} now follows.
\end{proof}

Lemma \ref{T:interwining} allows us to express convolutions in terms of Fourier expansions. In fact, by \eqref{E:convolution-fourier}, $C_S$ is scalar multiplication by $\widehat S(l)$ on each $\mathcal H_l$, and thus
\begin{equation}\label{E:convdiagonalization}
f=\sum_{l\geq 0}\sum_{m=1}^{d(l)} c_{lm}\,Y_{lm}\implies C_S(f)=\sum_{l\geq 0} \sum_{m=1}^{d(l)} c_{lm}\, \widehat S(l) Y_{lm}
\end{equation}
for all $f\in\mathscr E(S^n)$.  (Since $C_S$ commutes with left translations from $\mathrm{SO}(n)$, it preserves each $\mathcal H_l$ and acts as scalar multiplication there.  The Schur constants are precisely the $\widehat S(l)$.)

It will now be useful for us to state some important mapping properties of the convolution operator $C_S$ in terms of the coefficients $\widehat S(l)$.

\begin{theorem}\label{T:spherconvproperties}
Let $S\in\mathscr E'_K(S^n)$. The right convolution operator $C_S:\mathscr E(S^n)\to\mathscr E(S^n)$ is injective if and only if $\widehat S(l)\neq 0$ for all $l$.  The operator $C_S$ is surjective if and only if there exists a nonnegative integer $M$ and a constant $C>0$ such that
\begin{equation}\label{E:lowerestimate}
|\widehat S(l)|\geq C\,(1+l)^{-M}
\end{equation}
for all $l\geq 0$.
\end{theorem}
\begin{proof}
The first assertion is obvious from \eqref{E:convdiagonalization}.

As for the second assertion, suppose that $\widehat S(l)$ satisfies the slow decay condition \eqref{E:lowerestimate}.  If $h\in\mathscr E(S^n)$, with $h=\sum_l\sum_m d_{lm}\,Y_{lm}$, then the Fourier coefficients $d_{lm}$ are rapidly decreasing in $l$, and \eqref{E:lowerestimate} shows that
$\widehat S(l)^{-1}d_{lm}$ is still rapidly decreasing in $l$.  If we put
$$
f=\sum_{l\geq 0}\sum_{m=1}^{d(l)} \widehat S(l)^{-1}d_{lm}\,Y_{lm},
$$
then $f\in\mathscr E(S^n)$ and \eqref{E:convdiagonalization} shows that $C_S(f)=h$.

Conversely, suppose that $C_S$ is surjective.  The relation \eqref{E:convdiagonalization} shows that $\widehat S(l)$ can never be zero, so $C_S$ is injective.  Since $\mathscr E(S^n)$ is a Fr\'echet space, the open mapping theorem therefore implies that $C_S$ is a homeomorphism.
Since evaluation at $o$ is a continuous linear functional on $\mathscr E(S^n)$, Theorem \ref{T:spheresmoothtopology}  implies  that there is a positive constant $B$ and a nonnegative integer $N$ such that
$$
|(C_S^{-1} f)(o)|\leq B\sum_{k=0}^N \|(\Delta_{S^n})^k f\|_\infty.
$$ 
for any $f\in\mathscr E(S^n)$.  Now on each $\mathcal H_l$, $C_S^{-1}$ is multiplication by the scalar $(\widehat S(l))^{-1}$.  In the above inequality, if we specialize $f$ to the spherical function $\varphi_l$ and note that $\|\varphi_l\|_\infty=\varphi_l(o)= 1$, we obtain
$$
\left|(\widehat S(l))^{-1}\right|\leq B\,\sum_{k=0}^N (l(l+n-1))^k\leq BN\,(l(l+n-1))^N.
$$
The estimate \eqref{E:lowerestimate} now follows, with $M=2N$.
\end{proof}

This completes our preliminary discussion.  For the rest of the section our focus will be the study of  solutions of 
the following shifted wave equation on  $S^n$:
\begin{equation}\label{E:modified-wave-sphere}
\left(\Delta_{S^n}-\left(\frac{n-1}{2}\right)^2\right) u(x,t)=\frac{\partial^2}{\partial t^2}\,u(x,t),
\end{equation}
where $u(x,t)\in C^\infty(S^n\times\mathbb R)$, where $\Delta_{S^n}$  acts on the first argument $x\in S^n$, and where $t$ denotes time.  This is the elliptic space analogue of the shifted wave equation on hyperbolic space $\mathbb H^n$ considered, for example, in \cite{GGA}, Chapter II, which can be solved using Asgeirsson's mean value theorem. (This latter shifted wave equation on $\mathbb H^n$ is a special case of the equation \eqref{E:modifiedwaveeqn} considered in \S4.)

As was mentioned earlier, equation \eqref{E:modified-wave-sphere} was studied in 1978 by Lax and Phillips \cite{LaxPhillips1978}, who used the easily obtained series solution to deduce Huygens' principle.  A conformal change of variables was then used to transform this equation to the standard wave equation in $\mathbb R^n$, from which the Euclidean Huygens' principle follows.  Equation \eqref{E:modified-wave-sphere} is an example of a rank one ``multitemporal wave equation'' on a compact symmetric space.  

Equation (\ref{E:modified-wave-sphere}) was also studied by Gonzalez and Zhang in 2006 \cite{GonzalezZhang2006}, where they obtained explicit mean value solutions for given initial conditions, and from which Huygens' principle is an obvious consequence.

If we assume the initial conditions
\begin{equation}\label{E:sherewaveinitconds}
\begin{split}
u(x,0)&=f_0(x)\\
\partial_t u(x,0)&=g(x),
\end{split}
\end{equation}
where $f_0,\,g\in C^\infty(S^n)$, we can easily obtain the series solution in \cite{LaxPhillips1978} as follows. 

Express the functions $f_0$ and $g$ in spherical harmonics by
\begin{align}
f_0(x)&=\sum_{l\geq 0} \sum_{m=1}^{d(l)} a_{lm}\,Y_{lm}(x)\label{E:sphericalinitpos}\\
g(x)&=\sum_{l\geq 0} \sum_{m=1}^{d(l)} b_{lm}\,Y_{lm}(x), \label{E:spherinitvel}
\end{align}
where $a_{lm}=\langle f_0,Y_{lm}\rangle$ and $b_{lm}=\langle g, Y_{lm}\rangle$. We note that the coefficients $a_{lm}$ and $b_{lm}$ are rapidly decreasing in $l$.

Now suppose that $u(x,t)\in C^\infty(S^n\times\mathbb R)$ solves the initial value problem \eqref{E:modified-wave-sphere} -- \eqref{E:spherinitvel}.  We can expand $u$ in spherical harmonics by
$$
u(x,t)=\sum_{l\geq 0} \sum_{m=1}^{d(l)} g_{lm}(t)\,Y_{lm}(x).
$$
Using \eqref{E:harmoniceigenfcn}, a separation of variables calculation leads quickly to the following series expansion for the solution $u(x,t)$:
\begin{equation}\label{E:modifiedwavespheresoln}
u(x,t)=\sum_{l\geq 0} \sum_{m=1}^{d(l)} \left( a_{lm}\,\cos\left(l+\frac{n-1}{2}\right) t+b_{lm}\,\frac{\sin(l+\frac{n-1}{2})t}{(l+\frac{n-1}{2})} \right)\,Y_{lm}(x).
\end{equation}
We observe that $u(x,t)$ is periodic in $t$ of period $4\pi$ when $n$ is even and $2\pi$ when $n$ is odd.

When $n$ is odd, $(n-1)/2$ is an integer, and Lax and Phillips
deduced Huygens' principle as follows.  Since
$Y_{lm}(-x)=(-1)^l\,Y_{lm}(x)$, \eqref{E:modifiedwavespheresoln} gives
$$u(-x,t+\pi)=(-1)^{\frac{n-1}{2}}\,u(x,t).$$  
Suppose that the initial data $f_1(x)$
and $f_2(x)$ are supported on a spherical cap $B_\epsilon(o)$, where (again)
$o$ denotes the north pole.  Then when $t=\pi$, the
solution $u(x,\pi)$ is supported in $B_\epsilon(-o)$ ($-o$ denoting
the south pole).  Since the shifted wave equation \eqref{E:modified-wave-sphere} is hyperbolic, the solution
propagates at unit speed forwards and backwards in time.  Thus, for any $t$ such that
$0<t<\pi$, the solution $x\mapsto u(x,t)$ is supported in both $B_{t+\epsilon}(o)$ and
$B_{\pi-t+\epsilon}(-o)$; the intersection of these two spherical caps is the band
$t-\epsilon<d(o,x) < t+\epsilon$.

We can also express the solution $u(x,t)$  in terms of convolutions on $S^n$.  
For each $t\in\mathbb R$, we observe that the coefficients $\cos\left(l+\frac{n-1}{2}\right)t$ and $\frac{\sin(l+(n-1)/2)t}{l+(n-1)/2}$ in \eqref{E:modifiedwavespheresoln} are bounded functions of $l$. By the Paley-Wiener Theorem (see \eqref{E:sphere-paleywiener} and the discussion right after it), there must exist unique distributions $S_t$ and $S_t'$ in $S^n$ in $\mathscr E'_K(S^n)$ such that
\begin{equation}\label{E:spherewavefundsoln}
\widehat S_t(l)=\frac{\sin(l+\frac{n-1}{2}) t}{l+\frac{n-1}{2}},\quad \widehat S_t'(l)=\cos\left(l+\frac{n-1}{2}\right)t,\qquad l\in\mathbb Z^+.
\end{equation}
Relations \eqref{E:convdiagonalization} and \eqref{E:modifiedwavespheresoln} then show that
\begin{equation}\label{E:sphereconvsoln}
u(x,t)=f_0*S'_t(x)+g*S_t(x).
\end{equation}
If $|t|<\pi$, it can be shown (\cite{OlafssonSchlichtkrull2008}) that the distributions $S_t$ and $S_t'$ are supported in the spherical cap $B_{|t|}(o)$, but we will not need this result. 

Here is the point where we can now discuss the snapshot problem for the shifted wave equation on $S^n$.  For convenience we call any $C^\infty$ solution $u(x,t)$ to \eqref{E:modified-wave-sphere} a \emph{wave} on $S^n$.  For any $t\in \mathbb R$ let us put $u_t(x)=u(x,t)$ and call $u_t$ the \emph{snapshot} of $u$ at time $t$.

Because waves on $S^n$ are periodic in $t$, it seems natural to expect that questions of existence and uniqueness related to the snapshot problem will already be meaningful when we consider just two snapshots of a given wave.

Let us first consider uniqueness. Fix any real number $\alpha\neq 0$, and suppose $u(x,t)$ is a wave on $S^n$ in which we are given the snapshots $u_0$ and $u_\alpha$ at times $t=0$ and $\alpha$, respectively.  By replacing $t$ by $t-\alpha$ if necessary, we can assume that $\alpha>0$. The series solution \eqref{E:modifiedwavespheresoln} then allows us to produce a formula for the snapshots $u_{m\alpha}$ in terms of $u_0$ and $u_\alpha$ for all $m\in\mathbb Z$.

To see this, let us first observe that for each $m\in\mathbb Z$, there is a $K$-invariant distribution $\Psi_{m,\alpha}\in\mathscr E'(S^n)$ such that
$$
\widehat \Psi_{m,\alpha}(l)=U_{m-1}\left(\cos \left(l+\frac{n-1}{2}\right)\alpha\right)=\frac{\sin(l+\frac{n-1}{2})m\alpha}{\sin(l+\frac{n-1}{2})\alpha}.
$$
This is because $U_{m-1}\left(\cos\left(l+\frac{n-1}{2}\right)\alpha \right)$, being a polynomial in the cosine, is a bounded function of $l$.  (This is also clear from looking at the right hand side above.) Hence \eqref{E:modifiedwavespheresoln} or \eqref{E:sphereconvsoln} shows that
\begin{equation}\label{E:modifiedwaverecusion}
u_{m\alpha}(x)=u_0*S'_{m\alpha} + (u_\alpha-u_0*S'_\alpha)*\Psi_{m,\alpha}.
\end{equation}
Again we recall that the wave $u(x,t)$ is periodic in $t$ of period $2\pi$ or $4\pi$, depending on whether $n$ is odd or even.  If $\alpha/\pi$ is irrational, the set of congruence classes of $m\alpha$ mod $2\pi$ or $4\pi$ is dense in $[0,2\pi]$ or $[0,4\pi]$, respectively.  By continuity, this shows that the snapshots $u_{m\alpha}$ determine the wave $u(x,t)$ uniquely.

We can also infer uniqueness from the series solution \eqref{E:modifiedwavespheresoln}.  Again suppose that we are given the snapshots $u_0$ and $u_\alpha$ of $u(x,t)$. Then the Fourier coefficients $a_{lm}$ of $u_0$ are uniquely determined.  The formula  \eqref{E:modifiedwavespheresoln} then shows that each of the coefficients $b_{lm}$ of the initial velocity $\partial_t u(x,0)$ can be uniquely determined from the formula for $u_\alpha$ \emph{provided that}  $\sin (l+\frac{n-1}{2})\alpha$ is never zero for any $l$.  (Once the coefficients $b_{lm}$ are known, the formula \eqref{E:modifiedwavespheresoln} will give us the wave $u(x,t)$.)  This will certainly be the case when $\alpha/\pi$ is irrational.

Even when $\alpha/\pi$ is rational, the coefficient $\sin (l+\frac{n-1}{2})\alpha$ is never zero in certain cases.  To see this, write 
$\alpha/\pi$ as a fraction $p/q$ in lowest terms, with $q>0$.  In case $n$ is even and $p$ is odd, the product $(l+\frac{n-1}{2})p$ is always a half integer, so $(l+\frac{n-1}{2})p\pi/q$ is never an integer multiple of $\pi$. Hence
$$
\sin \left(\left(l+\frac{n-1}{2}\right)\frac{p\pi}{q}\right)\neq 0 \quad\text{ for all }l.
$$
Thus for $n$ even and $p$ odd we can recover $b_{lm}$ for each $l$ and $m$, and hence $u$ is uniquely determined.

By contrast, when $n$ is odd or $p$ is even, the product $(l+\frac{n-1}{2})p$ is an integer which is divisible by $q$ for infinitely many $l$. For such $l$, we have $\sin\left((l+\frac{n-1}{2})p\pi/q\right)=0$, so $b_{lm}$ can never be determined.  In addition, for these $l$, there are uncountably many possible ways to choose $b_{lm}$ so as to make the Fourier coefficients in the series \eqref{E:spherinitvel} rapidly decreasing.

We have thus arrived at the following result.

\begin{theorem}\label{T:snapshot-sphere-uniqueness}
Fix a real number $\alpha>0$.  A wave $u(x,t)$ on $S^n$ is uniquely determined by its snapshots $u_0$ and $u_\alpha$ provided that either
\begin{enumerate}
\item $\alpha/\pi$ is irrational, or
\item $\alpha/\pi$ is a rational number $p/q$ in lowest terms, and $n$ is even and $p$ is odd.
\end{enumerate}
On the other hand, suppose that $\alpha/\pi$ is a rational number $p/q$ in lowest terms, and with either $p$ even or $n$ odd. Then there are uncountably many waves $u(x,t)$ with snapshots $u_0$ and $u_\alpha$.
\end{theorem}

Conclusion (2) above seems rather surprising, since whenever $\alpha/\pi=p/q$, there are only finitely many congruence classes of $m\alpha$ modulo $4\pi$, so the formula \eqref{E:modifiedwaverecusion} will ostensibly give us only finitely many distinct snapshots  $u_{m\alpha}$ of the wave $u$ from the data giving $u_0$ and $u_\alpha$. 

Finally, we consider existence.  Let $\alpha$ be a positive real number, and suppose that we are given two functions $f_0$ and $f_\alpha$ in $\mathscr E(S^n)$.  We wish to determine whether there exists a wave $u(x,t)$ on $S^n$ such that $u_0=f_0$ and $u_\alpha=f_\alpha$.  It turns out that this will depend on whether $\alpha/\pi$ is a Liouville number, and for $n$ even, on whether $\alpha/\pi$ is a certain \emph{type} of Liouville number.  We now clarify the type of Liouville number we have in mind.

If $\beta$ is a real number, let us call $\beta$ a \emph{Liouville number of odd type} provided that for any integer $N>0$, there exists a fraction $p/q$, with $q$ \emph{odd} and $>1$, such that
the Liouville condition \eqref{E:liouville-def}, reprised below, holds:
$$
0<\left|\beta -\frac{p}{q}\right| < \frac{1}{q^N}.
$$
This means that $\beta$ can be rapidly approximated by fractions with odd denominators.

We now state some facts concerning Liouville numbers of odd type.  First we note that if $\beta$ is \emph{any} Liouville number, then for any $N$, the estimate above is satisfied for some fraction $p/q$ where the denominator $q$ is even.  (Thus it doesn't make sense to define Liouville numbers of even type.) In fact, given $N$, we know that there is a fraction $p_1/q_1$, with $q_1\geq 2$, such that
$$
0<\left|\beta -\frac{p_1}{q_1}\right| < \frac{1}{q_1^{2N}}.
$$
Then
$$
0<\left|\beta -\frac{2p_1}{2q_1}\right| < \left(\frac{4}{2q_1}\right)^N\cdot\frac{1}{(2q_1)^N} \leq\frac{1}{(2q_1)^N}.
$$
Next, we observe that the set of all Liouville numbers $\beta$ of odd type is uncountable and dense in $\mathbb R$. To see this, consider any irrational number in $[0,1]$ whose ternary expansion is 
\begin{equation}\label{E:ternarybeta}
\beta=\sum_{j=1}^\infty \frac{c_j}{3^{j!}},
\end{equation}
where each $c_j$ belongs to $\{0,1,2\}$.  For this $\beta$, we note that for any $N\geq 1$, the partial sum $\sum_{j=1}^N c_j\,3^{-j!}$ is a fraction of the form $p/3^{N!}$.
Thus
$$
0<\beta-\frac{p}{3^{N!}}=\sum_{j=N+1}^\infty \frac{c_j}{3^{j!}}<\sum_{l=(N+1)!}^\infty \frac{2}{3^l}=\frac{2}{3^{(N+1)!}}\cdot\frac 32<\frac{1}{(3^{N!})^N}
$$
Thus $\beta$ is a Liouville number of odd type.  The set of all irrational $\beta$ of the form \eqref{E:ternarybeta} is an uncountable subset of the closed interval $[0,1]$. Now any sum of the form $\beta+j/3^k$, where  $\beta$ is an irrational number of the kind given above, and $j$ and $k$ are integers with $k\geq 0$, is also a Liouville number of odd type, and the set of such sums is uncountable and dense in $\mathbb R$.  (Nonetheless, they form a set of measure zero, since the set of all Liouville numbers has measure zero.)

Finally, our definition of Liouville numbers of odd type is meaningful since the following proposition shows that not all Liouville numbers are of odd type.
\begin{proposition}\label{T:notoddtype}
Let 
$$
\beta=\sum_{j=1}^\infty \frac{1}{2^{j!}}.
$$
Then $\beta$ is a Liouville number which is not of odd type.
\end{proposition}

The fact that this $\beta$ is a Liouville number can be proved in the same way as the $\beta$ given in the ternary expansion \eqref{E:ternarybeta}.  We will postpone the rather technical \hyperlink{proofof5.6}{proof} of the rest of this proposition to the end of the section.

Here is our main theorem regarding the existence of waves in $S^n$ with snapshots at two given times.

\begin{theorem}\label{T:sphere-snapshot-existence}
Fix a real number $\alpha>0$, and suppose that we are given two functions $f_0$ and $f_\alpha$ in $\mathscr E(S^n)$.  Then there exists a wave $u(x,t)$ on $S^n$ with snapshots $u_0=f_0$ and $u_\alpha=f_\alpha$ if and only if the following condition holds:
\begin{equation}\label{E:sphere-snapshot-compatibility}
f_\alpha-f_0*S'_\alpha\in \mathscr E(S^n)* S_\alpha.
\end{equation}
Let $\beta=\alpha/\pi$.
\begin{enumerate}
\item If $n$ is odd, Condition \eqref{E:sphere-snapshot-compatibility} will hold for all pairs $(f_0,\,f_\alpha)$ if and only if $\beta$ is irrational and not a Liouville number.
\item If $n$ is even, Condition \eqref{E:sphere-snapshot-compatibility} will hold for all pairs $(f_0,\,f_\alpha)$ if and only if $\beta$ is one of the following types of numbers:
\begin{itemize}
\item $\beta$ is a rational number $p/q$ in lowest terms, with $p$ odd, or
\item $\beta$ is irrational and $\beta/2$ is not a Liouville number of odd type.
\end{itemize}
\end{enumerate}
In all the cases for $\beta$ enumerated above, the wave $u(x,t)$ with snapshots $(f_0,\,f_\alpha)$ is unique.
\end{theorem}

\begin{remark}
\hfill
\begin{enumerate}
\item[(a)] This theorem therefore implies that for almost all $\alpha$, there is a unique solution to the snapshot problem for any given pair $(f_0,f_\alpha)$ in $\mathscr E(S^n)$.
\item[(b)] If $\beta/2$ is a Liouville number of odd type, then so is $\beta$.  However, if $\beta$ is a Liouville number of odd type, there is no guarantee that $\beta/2$ is Liouville of odd type.
\item[(c)] In the case of even-dimensional spheres, it is surprising (to us) that the snapshot problem has unique solutions even for certain rational or Liouville $\beta$.
\end{enumerate}
\end{remark}

\begin{proof}
The first assertion \eqref{E:sphere-snapshot-compatibility} of the theorem is an immediate consequence of the convolution solution \eqref{E:sphereconvsoln}.

For the second assertion, it is clear that the condition \eqref{E:sphere-snapshot-compatibility} holds for all pairs $(f_0,f_\alpha)$ if and only if $C_{S_\alpha}(\mathscr E(S^n))=\mathscr E(S^n)$.  By Theorem \ref{T:spherconvproperties} and relation \eqref{E:spherewavefundsoln}, the set of all  $\alpha>0$ such that $C_{S_\alpha}$ is surjective on $\mathscr E(S^n)$ is precisely the set of all $\alpha>0$ for which the function
$$
\widehat S_\alpha(l)= \frac{\sin (l+\frac{n-1}{2})\alpha}{l+\frac{n-1}{2}},\qquad\qquad l\in\mathbb Z^+
$$
satisfies the slow decay condition \eqref{E:lowerestimate}.  This set is clearly the same as the set of all $\alpha>0$ for which the sequence
\begin{equation}\label{E:sine-coeff}
l\mapsto \sin \left(l+\frac{n-1}{2}\right)\alpha=\sin \left(l+\frac{n-1}{2}\right)\beta\pi
\end{equation}
satisfies \eqref{E:lowerestimate}.  Recall that the surjectivity of $C_{S_\alpha}$ implies its injectivity, so any $\alpha$ such that the sequence \eqref{E:sine-coeff} satisfies \eqref{E:lowerestimate} must give rise to a unique wave $u(x,t)$ corresponding to every pair of snapshots $(f_0,\,f_\alpha)$.

\emph{Case 1: $n$ is odd.}  Then $m=(n-1)/2$ is an integer.

If $\beta$ is rational, then $(l+m)\beta$ is an integer for infinitely many $l$, so the sequence \eqref{E:sine-coeff} is zero  for infinitely many $l$, and hence it never satisfies \eqref{E:lowerestimate}.  

Next let us assume that $\beta$ is irrational.

Let $[\;\cdot\;]$ denote the \emph{nearest integer} function on $\mathbb R$.  (The exact definition is not important: this function assigns to each $t\in\mathbb R$ the integer nearest $t$ if $t$ is not a half integer, and one of the nearest two integers in case $t$ is a half integer.)

If $\beta$ is a Liouville number, then there exists a sequence $l_1<l_2<\cdots$ in $\mathbb Z^+$ such that
$$
\left| \left(l_k+m\right)\beta-\left[\left(l_k+m\right)\beta\right]\right| < (l_k+m)^{-k} 
$$
for all positive integers $k$.    Hence
\begin{align}\label{E:n-odd-beta-liouville}
\left|\sin \left\{(l_k+m)\beta\pi\right\}\right|&=\left|\sin \left\{\left(l_k+m\right)\beta-\left[\left(l_k+m\right)\beta\right]\right\}\pi\right|\nonumber \\
&\leq \left| \left(l_k+m\right)\beta-\left[\left(l_k+m\right)\beta\right]\right|\,\pi\nonumber \\
&\leq (l_k+m)^{-k}\,\pi.
\end{align}
Thus if $\beta$ is Liouville, the sequence \eqref{E:sine-coeff} can never be slowly decaying.

If $\beta$ is not a Liouville number (and is irrational), then there is an integer $N>0$ such that
$$
|(l+m)\beta-[(l+m)\beta]|\geq (l+m)^{-N}
$$
for all $l\in\mathbb Z^+$.  Now the terms on the left are always $\leq 1/2$, so
\begin{align}\label{E:n-odd-beta-not-liouville}
\left|\sin \left\{(l+m)\beta\pi\right\}\right|&=\left|\sin\left\{((l+m)\beta-[(l+m)\beta])\pi\right\}\right|\nonumber \\
&\geq \frac 12\,\left|(l+m)\beta-[(l+m)\beta]\right|\pi\nonumber \\
&\geq\frac{\pi}{2}\,(l+m)^{-N}.
\end{align}
For this $\beta$, the sequence \eqref{E:sine-coeff} therefore satisfies the decay condition \eqref{E:lowerestimate}.

This covers all the possibilities for $\beta$ when $n$ is odd, and proves Conclusion (1) in the statement of Theorem \ref{T:sphere-snapshot-existence}.

\emph{Case 2: $n$ is even.}  Then $m=n-1$ is an odd integer.

Assume first that $\beta$ is a rational number $p/q$ in lowest terms.  If $p$ is even, then $q$ is odd and the sequence
\begin{equation}\label{E:sine-sequence}
l\mapsto \sin\left(\left(l+\frac{m}{2}\right)\frac{p\pi}{q}\right)
\end{equation}
is zero for infinitely many $l$, so it cannot be slowly decaying.  If $p$ is odd, the sequence \eqref{E:sine-sequence} is periodic (of period $2q$) and never zero, so is slowly decaying.

Next assume that $\beta$ is irrational.  If $\beta/2$ is a Liouville number of odd type, then there is a sequence $l_1<l_2<\cdots$ of nonnegative integers such that
$$
\left|(2l_k+m)(\beta/2) - [(2l_k+m)(\beta/2)]\right|< (2l_k+m)^{-k}
$$
for all $k$.  Hence just as wth the inequality \eqref{E:n-odd-beta-liouville} in Case 1, we obtain
$$
\left|\sin\{(2l_k+m)(\beta\pi/2)\}\right|<(m+2l_k)^{-k}\,\pi.
$$
For these $\beta$, the sequence \eqref{E:sine-coeff} therefore cannot be slowly decaying.

Finally, let us assume that $\beta/2$ is not a Liouville number of odd type (and is irrational).  Then there is a positive integer $N$ such that
$$
\left|(2l+m)(\beta/2) - [(2l+m)(\beta/2)]\right|\geq(2l+m)^{-N}
$$
for all $l\in\mathbb Z^+$. Just as with inequality \eqref{E:n-odd-beta-not-liouville} in Case 1, we can again argue that
$$
\left|\sin\{(2l+m)(\beta\pi/2)\}\right|\geq\frac{\pi}{2}\,(2l+m)^{-N}
$$
for all $l$.  For these $\beta$, the sequence \eqref{E:sine-coeff} is therefore slowly decaying.

We have covered all the possible cases for $\beta$ when $n$ is even.  This completes the proof of Theorem \ref{T:sphere-snapshot-existence}.
\end{proof}

We finish the section by providing a proof of Proposition \ref{T:notoddtype}.

\hypertarget{proofof5.6}{\emph{Proof of Proposition \ref{T:notoddtype}:}}  We wish to show that the Liouville number $\beta=\sum_{j=1}^\infty 2^{-j!}$ is not of odd type.

Recall that $[t]$ denotes the nearest integer to $t$. The proposition will follow if we can prove that
\begin{equation}\label{E:liouville-est}
|q\beta-[q\beta]| > q^{-3}
\end{equation}
for all sufficiently large odd integers $q$.

To prove \eqref{E:liouville-est}, suppose that $q$ is any odd integer greater than $2^{3!}$.  Since $q$ is odd, we can write $q$ uniquely as
\begin{equation}\label{E:qbinaryrep}
q=2^m+a_{m-1}\cdot 2^{m-1}+ \cdots +a_1\cdot 2+1,
\end{equation}
for some positive integer $m\geq 3!$, and where $a_j\in\{0,1\}$ for $j=1,\ldots,m-1$.  Let $k$ be the positive integer such that $k!\leq m <(k+1)!$.  We will now consider two cases for $m$.

\emph{Case 1:} $m\geq k\cdot k! - 2$.

For this $m$ we have
\begin{align}
q\beta&=\sum_{j=1}^k q\cdot 2^{-j!} + q\cdot 2^{-(k+1)!} + \sum_{j=k+2}^\infty q \cdot 2^{-j!}\nonumber\\
&=\sum_{j=1}^k q\cdot 2^{-j!} + \left( 2^{m-(k+1)!} +\sum_{s=1}^{m-1} a_s\cdot 2^{s-(k+1)!}\right)\nonumber\\
&\qquad\qquad\qquad\qquad\qquad\qquad + 2^{-(k+1)!} + \sum_{j=k+2}^\infty q\cdot 2^{-j!}\nonumber\\
&=p+\sum_{r=1}^{(k+1)!-1} d_r \cdot 2^{-r} + 2^{-(k+1)!} + \sum_{j=k+2}^\infty q\cdot 2^{-j!},\label{E:q-beta1}
\end{align}
where $p$ is an integer and the $d_r$'s belong to $\{0,1\}$.

Since $q<2^{m+1}\leq 2^{(k+1)!}$, we observe that
\begin{equation}\label{E:q-beta2}
\sum_{j=k+2}^\infty q\cdot 2^{-j!} < \sum_{j=k+2}^\infty 2^{m+1-j!}\leq \sum_{j=k+2}^\infty 2^{(k+1)!-j!}<\sum_{l=(k+1)(k+1)!}^\infty 2^{-l}.
\end{equation}
This inequality and \eqref{E:q-beta1} then show that the integer $p$ in \eqref{E:q-beta1} is in fact equal to the floor function $\floor{q\beta}$.  

Now if the nearest integer $[q\beta]$ is the floor $\floor{q\beta}$, then \eqref{E:q-beta1} shows that
$|q\beta-[q\beta]|>2^{-(k+1)!}$.  On the other hand, if $[q\beta]$ equals the ceiling $\ceil{q\beta}$, then \eqref{E:q-beta1} and \eqref{E:q-beta2} show that
\begin{align*}
|q\beta-[q\beta]|&=1-\left(\sum_{j=1}^{(k+1)!-1} d_r \cdot 2^{-r} + 2^{-(k+1)!} + \sum_{j=k+2}^\infty q\cdot 2^{-j!}\right)\\
& > 1-\left(\sum_{j=1}^{(k+1)!-1} d_r \cdot 2^{-r} + 2^{-(k+1)!} + \sum_{l=(k+1)(k+1)!}^\infty 2^{-l}\right)\\
& >  2^{-(k+1)!-1}.
\end{align*}
In either case, we therefore have $|q\beta - [q\beta]|>2^{-(k+1)!-1}$.  Now since we have assumed that $m\geq k\cdot k! -2$, we see that $q> 2^m\geq 2^{k\cdot k!-2}$, and thus
\begin{align*}
q^{1/k}  > 2^{k!}\cdot 2^{-2/k} &\implies q^{-(k+1)/k} < 2^{-(k+1)!}\cdot 2^{2(k+1)/k}\\
&\implies 2^{-(k+1)!-1} > q^{-1-1/k}\cdot 2^{-2}\cdot 2^{-2/k}\cdot 2^{-1}\\
&\implies 2^{-(k+1)!-1} > q^{-2}\cdot 2^{-4} > q^{-3}.
\end{align*}
It then follows that $|q\beta-[q\beta]|>q^{-3}$.
\medskip

\emph{Case 2:} $m < k\cdot k!-2$.

From the expression \eqref{E:qbinaryrep} for $q$, we obtain
\begin{align}
q\beta&=\sum_{j=1}^{k-1} q\cdot 2^{-j!} + q\cdot 2^{-k!} +\sum_{j=k+1}^\infty q\cdot 2^{-j!}\nonumber \\
&= p + \sum_{s=1}^{k!-1} d_s\cdot 2^{-s} + 2^{-k!} + \sum_{j=k+1}^\infty q\cdot 2^{-j!},\label{E:q-beta3}
\end{align}
where again $p$ is an integer and the $d_s$ belong to $\{0,1\}$.  Now the assumption on $m$ in the present case implies that for any $j\geq k+1$,
$$
m+1-j!\leq m+1-(k+1)! \leq k\cdot k!-2-(k+1)!=-k!-2.
$$
Since $q<2^{m+1}$, this implies that
\begin{equation}\label{E:q-beta4}
\sum_{j=k+1}^\infty q\cdot 2^{-j!} < \sum_{j=k+1}^\infty 2^{m+1-j!} <\sum_{r=k!+2}^\infty 2^{-r}.
\end{equation}
This inequality and \eqref{E:q-beta3} thus show that the integer $p$ in \eqref{E:q-beta3} is in fact the floor function $\floor{q\beta}$.  Now if the nearest integer $[q\beta]$ happens to be the floor $p=\floor{q\beta}$, then
\eqref{E:q-beta3} implies that
$$
|q\beta -[q\beta]|>2^{-k!}.
$$
If the nearest integer $[q\beta]$ equals the ceiling $\ceil{q\beta}$, then \eqref{E:q-beta3} and \eqref{E:q-beta4} imply that
\begin{align*}
|q\beta -[q\beta]|& = 1-\left(\sum_{s=1}^{k!-1} d_s\cdot 2^{-s} + 2^{-k!} + \sum_{j=k+1}^\infty q\cdot 2^{-j!}\right)\\
& > 1 - \left(\sum_{s=1}^{k!-1} d_s\cdot 2^{-s} + 2^{-k!} + \sum_{r=k!+2}^\infty 2^{-r}\right)\\
& \geq 2^{-k!-1}.
\end{align*}
Either way, we have $|q\beta -[q\beta]| > 2^{-k!-1}$.  Now $q > 2^m\geq 2^{k!}$, so
$$
2^{-k!-1} > 2^{-1}\cdot q^{-1} > q^{-2}.
$$
It follows that
$$
|q\beta -[q\beta]| > q^{-2} > q^{-3},
$$
proving our claim \eqref{E:liouville-est} in this case.  This finishes the proof of Proposition \ref{T:notoddtype}.
\qed

\emph{Remark.} The shifted wave equation \eqref{E:modified-wave-sphere} was also studied by B. Rubin \cite{RubinSmallDenom}, who considered the ``shifted Cauchy data'':  $(\partial_t u)(x,0)=0$ and  $u(x,\alpha)=f_\alpha(x)$, where $f_\alpha$ belongs to the Sobolev space $\mathcal H^s(S^n)$.  One can see that this is a type of snapshot problem.  Considered in our context of smooth functions (so $s=\infty$), the convolution solution \eqref{E:sphereconvsoln} shows that the problem of existence reduces to the question of whether the convolution operator $C_{S'_\alpha}\colon\mathscr E(S^n)\to\mathscr E(S^n)$ is surjective. In Theorem 6.1 of that paper he obtains Liouville-type conditions on $\alpha/\pi$, when it is irrational, which are sufficient to guarantee the existence of solutions.

\bibliographystyle{alpha}

\bibliography{Gonzalez}
\end{document}